\documentclass[reqno,10pt]{amsart}
\usepackage{dsfont}
\usepackage{xspace}
\usepackage[bookmarksnumbered]{hyperref}

\newcommand{\Ker}{\mathrm{Ker}}
\newcommand{\de}{\mathrm{d}}                        
\newcommand{\N}{\ensuremath{\mathds{N}}\xspace}     
\newcommand{\R}{\ensuremath{\mathds{R}}\xspace}     
\newcommand{\prodot}[1]
    {\ensuremath{\left\langle#1\right\rangle}}      

\newcommand{\Hi}{\mathcal{H}}
\newcommand{\K}{\mathcal{K}}

\newcommand{\spa}{\mathrm{span}\,}

\newtheorem{theorem}{Theorem}[section]
\newtheorem{corollary}[theorem]{Corollary}
\newtheorem{lemma}[theorem]{Lemma}
\newtheorem{proposition}[theorem]{Proposition}
\theoremstyle{definition}
\newtheorem{definition}[theorem]{Definition}
\theoremstyle{remark}
\newtheorem{remark}[theorem]{Remark}
\title[Pseudo focal points]{Pseudo focal points along Lorentzian geodesics and Morse index}

\author[M. A. Javaloyes]{Miguel \'{A}ngel Javaloyes}
\address{Departamento de Geometr\'{\i}a y Topolog\'{\i}a.\hfill\break\indent
 Facultad de Ciencias, Universidad de Granada.\hfill\break\indent
 Campus Fuentenueva s/n, 18071 Granada, Spain}
 \email{majava@ugr.es}

\author[A. Masiello]{Antonio Masiello}
\address{Dipartimento di Matematica,\hfill\break\indent
Politecnico di Bari, \hfill\break\indent Via Orabona 4,
70125, Bari, Italy}
\email{masiello@poliba.it}

\author[P. Piccione]{Paolo Piccione}
\address{Departamento de Matem\'{a}tica,\hfill\break\indent
Universidade de S\~ao Paulo, \hfill\break\indent Rua do Mat\~ao 1010,
S\~ao Paulo, Brasil}
\email{piccione@ime.usp.br}

\date{April 18th, 2008}

\thanks{First author was partially supported  by Regional J.
Andaluc\'{\i}a Grant P06-FQM-01951, by Fundaci\'on S\'eneca project 04540/GERM/06   and by Spanish MEC Grant MTM2007-64504. The second author is supported by M.I.U.R. Research project PRIN07 "Metodi Variazionali e  
Topologici nello Studio di Fenomeni Nonlineari"
The third author is sponsored by Capes, Brazil, grant BEX 1509/08-0.}
\begin{document}
\begin{abstract}
Given a Lorentzian manifold $(M,g)$, a  geodesic $\gamma$ in $M$ and
a timelike Jacobi field $\mathcal Y$ along $\gamma$, we introduce a special class of instants along $\gamma$ that we call $\mathcal Y$-pseudo conjugate
(or focal relatively to some initial orthogonal submanifold).
We prove that the $\mathcal Y$-pseudo conjugate instants form a finite set, and their number equals the
Morse index of  (a suitable restriction of) the index form.
This gives a Riemannian-like Morse index theorem.
As special cases of the theory, we will consider geodesics in stationary
and static Lorentzian manifolds, where the Jacobi field $\mathcal Y$ is obtained as the
restriction of a globally defined timelike Killing vector field.
\end{abstract}

\maketitle

\begin{section}{Introduction}
In the last years, there have been several attempts of stating a Morse index  theorem for stationary Lorentzian manifolds.
Starting from the original results in  \cite{BGM98,BenMas92,GiMa95},
and aiming at establishing Morse theoretical results, several authors have studied the relations
between conjugate instants along a geodesic and its index form.
With the development of new functional analytical and symplectic techniques, it has appeared naturally
that the classical Riemannian statement of the theorem would not hold in the non positive definite case.
In first place, it is easy to prove that, unless the geodesic is Riemannian, the
index of its index form is always infinite. On the other hand, the conjugate instants along a semi-Riemannian geodesic,
unlike the Riemannian case,  may accumulate. As a matter of fact, there are several pathological examples where the set of conjugate
instants can be arbitrarily complicated (see \cite{PiTa03}).
In order to obtain a meaningful statement of the Morse index theorem, one has to replace the notion
of Morse index with the more general notion of \emph{spectral flow}, which is an integer number associated
to a continuous path of Fredholm symmetric bilinear forms. Moreover, the count of the conjugate instants has
to be interpreted as a suitable intersection number in the Grassmannian of all Lagrangian subspaces
in a finite dimensional symplectic space; this number is called \emph{Maslov index}.
The more general semi-Riemannian Morse index theorem (see for instance \cite{PiTa02}) states that,
given a semi-Riemannian manifold $(M,g)$ and a geodesic $\gamma:[0,1]\to M$, the spectral flow of
the paths of symmetric forms $I_t$, $t\in\left]0,1\right]$, obtained as restriction of the index
form of $\gamma$ to the set of variational vector fields along $\gamma\vert_{[0,t]}$, equals the
Maslov index of $\gamma$ up to the sign.

However, in the case of stationary Lorentzian manifolds, an alternative variational principle
is known for geodesics; among the main advantages, one proves that each one of its critical
points has finite index and, once again, its value equals the Maslov index of the corresponding geodesic.
This alternative variational principle will be described with more details below.
It is not known whether the set of conjugate instants along a given geodesic is discrete
in the stationary case. A very natural conjecture would be that, under the stationarity assumption,
conjugate instants do not accumulate, and that
the Maslov index of a geodesic is equal to their number, counted with multiplicity.
This conjecture still remains an open problem, although it has been proven to hold in some special cases.
For instance, in \cite{JaPi06a} the authors prove that this is true in the case
of semi-Riemannian Lie groups, endowed with a left-invariant metric, whose dimension is less than or equal to $5$.
Other than this special example, basically nothing is known concerning the distribution of conjugate instants along a geodesic
in a stationary manifold; the purpose of the present paper is to investigate in this direction.
\smallskip

In \cite{BenMas92} the authors establish a Riemannian-like Morse index theorem in a
static Lorentzian manifold by considering a functional on the Riemannian base. Due to a technical
gap in the proof, the result holds only under additional assumptions, although no counterexample to
their general statement has been found so far. Recently, the more general case
of stationary Lorentzian manifolds has been considered (see \cite{GiaMasPic01,GiaPic99}).
The central idea is to consider the energy functional restricted to the set of curves $\gamma:I\to M$
satisfying the natural constraint $g(\dot\gamma,\mathcal Y)=C_\gamma$, where $C_\gamma$ is a constant depending on
$\gamma$, $g$ is the Lorentzian metric on the stationary spacetime $M$ and $\mathcal Y$ is a timelike Killing
field in $(M,g)$. Such restriction has the same critical points as the original geodesic action
functional, but its second variation is \emph{essentially positive} at each critical
point. Thus, one has finite Morse index, and in \cite{GiaMasPic01} it is proven that this index
is equal to the Maslov index.\smallskip

The main goal of this paper is to study in more detail the distribution
of conjugate or focal instants, and to formulate a Riemannian-like Morse index theorem. Rather than restricting
to the fixed endpoints case, we will consider the more general case of geodesics whose endpoints
are free to vary along two given smooth submanifolds.
Our central result is the introduction of the class of $\mathcal Y$-pseudo conjugate, or $(\mathcal P,\mathcal Y)$-pseudo focal instants related to the choice of a timelike Jacobi field $\mathcal Y$;  these instants form a discrete set,
and they carry all the information about the second variation of the geodesic action functional up to a
correction term which is either null or equal to $1$. Although the notion of pseudo conjugate/focal point depends
on the (existence and the) choice of an everywhere timelike Jacobi field, in some specific situations
there is a canonical choice. This is the case, for instance, in stationary Lorentzian manifolds with a distinguished
timelike Killing vector field, which is the standard example we will refer to. 

Let us describe more precisely our result.
Consider a Lorentzian manifold $(M,g)$,  two smooth nondegenerate submanifolds $\mathcal P,\mathcal Q\subset M$ and
a geodesic $\gamma:[0,1]\to M$ with $\gamma(0)\in\mathcal P$, $\gamma(1)\in\mathcal Q$,
$\dot\gamma(0)\in T_{\gamma(0)}\mathcal P^\perp$ and $\dot\gamma(1)\in T_{\gamma(1)}\mathcal Q^\perp$. We will
call such a geodesic $\{\mathcal P,\mathcal Q\}$-orthogonal; let $\mathcal Y$ be a timelike Jacobi field along $\gamma$, for instance,
if $(M,g)$ is stationary, $\mathcal Y$ can be taken to be the restriction to $\gamma$ of a globally defined
timelike Killing vector field.
We will say that $\mathcal Y$ is \emph{admissible} (Definition~\ref{thm:defadmissible}) if $\mathcal Y(0)$ and the covariant derivative $\mathcal Y'(0)$ are  linearly independent. When the geodesic is spacelike or lightlike, given a $\mathcal Y$ non admissible, we can obtain an admissible timelike Jacobi field by perturbating the first one (see Lemma~\ref{perturbation}); on the other hand, $\mathcal Y$ will be called \emph{singular} if $\mathcal Y$ and $\mathcal Y'$ are everywhere
pointwise linearly dependent. For instance, if $(M,g)$ is static, and $\gamma$ is a geodesic orthogonal
to the static Killing vector field $\mathcal Y$, then the restriction of $\mathcal Y$ to $\gamma$ is singular.

A $({\mathcal P,\mathcal Y})$-pseudo Jacobi field $J$ is a smooth vector field along the geodesic $\gamma$
satisfying $g(J',\mathcal Y)-g(J,\mathcal Y')=0$ on $[0,1]$, such that
the vector field $J''-R[J]$ is a linear combination of $\mathcal Y$ and $\mathcal Y'$
and it satisfies certain initial conditions (see Definition~\ref{pseudo}). 
A $(\mathcal P,\mathcal Y)$-pseudo focal instant $t_0\in\left]0,1\right]$ is an 
instant such that there exists a $(\mathcal P,\mathcal Y)$-pseudo Jacobi field $J$ 
along $\gamma$ with $J(t_0)=0$. 
In this situation, one can consider the space $\mathcal H_0^{\{\gamma,\mathcal P,\mathcal Q\}}$
of all variational vector fields $V$ along $\gamma$ satisfying $V(0)\in T_{\gamma(0)}\mathcal P$,
$V(1)\in T_{\gamma(0)}\mathcal Q$ and the \emph{linear} constraint $g(V',\mathcal Y)-g(V,\mathcal Y')=0$ (we will suppress $\mathcal Q$ in the notations when $\mathcal Q$ reduces to a point); these correspond to variations of $\gamma$ by a smooth
family $\gamma_s$, $s\in\left]-\varepsilon,\varepsilon\right[$, by curves with $\gamma_s(0)\in\mathcal P$, 
$\gamma_s(1)\in\mathcal Q$ and such that the quantity $g(\dot\gamma_s,\mathcal Y)$ is constant and equal to the constant $C_\gamma=g(\dot\gamma,\mathcal Y)$
for all $s$. Such space has codimension one in the space $\mathcal H_*^{\{\gamma,\mathcal P,\mathcal Q\}}$ of all vector fields $V$ satisfying the more
general \emph{affine} constraint $g(V',\mathcal Y)-g(V,\mathcal Y')=\textrm{constant}$.

Consider the index form
$I_{\{\gamma,\mathcal P,\mathcal Q\}}$ given by the second variation at $\gamma$ of the geodesic action functional on the space
of curves with initial endpoint $\gamma(0)$ in $\mathcal P$ and final endpoint $\gamma(1)$ in $\mathcal Q$.
The difference between the indexes of the restrictions of $I_{\{\gamma,\mathcal P,\mathcal Q\}}$ to
the spaces $\mathcal H_0^{\{\gamma,\mathcal P,\mathcal Q\}}\times \mathcal H_0^{\{\gamma,\mathcal P,\mathcal Q\}}$ and
$\mathcal H_*^{\{\gamma,\mathcal P,\mathcal Q\}}\times \mathcal H_*^{\{\gamma,\mathcal P,\mathcal Q\}}$, which is at most one, is an invariant of the geodesic $\gamma$, that will be denoted
by $\epsilon_{\{\gamma,\mathcal P,\mathcal Q\}}$. It is an intriguing question to determine which geodesics have non vanishing
$\epsilon_{\{\gamma,\mathcal P,\mathcal Q\}}$, and how this fact affects the distribution of $\mathcal P$-focal instants
along $\gamma$. As a special example, we will consider the case of geodesics in static manifolds,
i.e., stationary manifolds whose Killing field $\mathcal Y$ has integrable orthogonal distribution.
In this case, each integral leaf of $\mathcal Y^\perp$ is a totally geodesic submanifold of $M$,
and those geodesics that are contained in one such integral submanifold have a purely
Riemannian behavior.

The main results of the paper are the following. First, we show that $(\mathcal P,\mathcal Y)$-pseudo focal
instants are related with the kernel of the restriction of the index form (Proposition~\ref{thm:prop2.6}).
The $(\mathcal P,\mathcal Y)$-pseudo focal instants form a finite
set, and their number equals the index of the restriction of  $I_{\{\gamma,\mathcal P\}}$ to
$\mathcal H_0^{\{\gamma,\mathcal P\}}\times \mathcal H_0^{\{\gamma,\mathcal P\}}$ (Morse index theorem,
Theorem~\ref{morsegeodesic}).
When we consider $I_{\{\gamma,\mathcal P,\mathcal Q\}}$ defined in
$\mathcal H_0^{\{\gamma,\mathcal P,\mathcal Q\}}\times \mathcal H_0^{\{\gamma,\mathcal P,\mathcal Q\}}$, we must add to the number of $(\mathcal P,\mathcal Y)$-pseudo focal instants, the index of a certain  symmetric bilinear form defined on a finite dimensional subspace. Moreover, in the singular
case the Morse index theorem holds in a stronger sense, in that the restrictions of
$I_{\{\gamma,\mathcal P\}}$ to $\mathcal H_*^{\{\gamma,\mathcal P\}}\times \mathcal H_*^{\{\gamma,\mathcal P\}}$ and to $\mathcal H_0^{\{\gamma,\mathcal P\}}\times \mathcal H_0^{\{\gamma,\mathcal P\}}$ have the same index,
i.e., $\epsilon_{\{\gamma,\mathcal P\}}=0$ (Theorem~\ref{morsesingulargeodesic}). The last result is applied to horizontal geodesics in static manifolds in Proposition~\ref{staticcase}.
A discussion on the distribution of pseudo focal and focal points along a geodesic is discussed in Section~\ref{sec:distribution}.
\smallskip

The proof of the main results is obtained by functional analytical techniques, involving the study of the
nullity and the variation of the index for a smooth family of Fredholm bilinear forms with varying domains.
Establishing the smoothness of the domains is a surprisingly non trivial fact (Proposition~\ref{c1}),
complicated by the occurrence of the singular case.
The kernel of the restriction of the index form $I_{\{\gamma,\mathcal P\}}$ to $\mathcal H_0^{\{\gamma,\mathcal P\}}\times \mathcal H_0^{\{\gamma,\mathcal P\}}$ is studied
in Section~\ref{sec:diffeq}. In order to get the Morse index theorem, in Section~\ref{abstractMorse} we prove
an abstract Morse Index Theorem in the spirit of \cite{Uhl73} (see also \cite{FrTh83,FrTh90}). As to the plethora of
abstract Morse index theorems appearing in the literature, few remarks are in order.
When dealing with a family of closed subspaces, it is customary to make two assumptions:
\begin{itemize}
\item \emph{monotonicity} of the family, to guarantee monotonicity of the index function;
\item \emph{continuity} of the family, to guarantee the semi-continuity of the index function.
\end{itemize}
These two assumptions are not totally independent; for instance, monotonicity is not compatible with
continuity in the norm operator topology (see Definition~\ref{thm:defcontfam} and Lemma~\ref{thm:lemcontinuity}). For the result aimed in this paper,
we cannot apply directly \cite[Theorem 1.11]{Uhl73}, because we cannot guarantee any kind of continuity
for our monotonic family of closed subspaces; however,  continuity in the norm operator topology is obtained
by considering a family of deformations (more precisely \emph{reparameterizations}, see Proposition~\ref{c1}) of the subspaces, but this operation
does not preserve monotonicity. The abstract index theorem proved here, Proposition~\ref{thm:MIT}, deals with this situation.

The authors gratefully acknowledge an important contribution to the final
version of the paper given by the anonymous referee, who pointed out a
mistake contained in the original version of the manuscript.
\end{section}
\begin{section}{An abstract Morse index theorem}\label{abstractMorse}
The main result of this section (Proposition~\ref{thm:MIT}) gives an abstract version of the Morse
index theorem for continuous families of bounded symmetric bilinear forms on varying domains.
Very likely, some of the preliminary results are already known in the literature, but for the reader's
convenience we give complete proofs of every statement. Basic bibliography for the topics
of this section are the classical textbooks \cite{Brezis83,Kato76}.

Let $H$ be a (real) Hilbert space, with inner product $\langle\cdot,\cdot\rangle$.
A bounded symmetric bilinear form $B:H\times H\to\R$ is said to be Fredholm if
it is represented by a (self-adjoint) Fredholm operator $T:H\to H$, i.e., $B=\langle T\cdot,\cdot\rangle$.
Note that the operator that represents $B$ depends on the choice of the inner product,
but the notion of Fredholmness does not.
A symmetric Fredholm bilinear form is \emph{nondegenerate} if $\Ker(B)=\{x\in H:B(x,y)=0\ \forall y\in H\}=\Ker(T)$
is trivial; this implies that $T$ is an isomorphism.
Observe that $\Ker(B)$ is finite dimensional if $B$ is Fredholm.
A subspace $Z\subset H$ is \emph{$B$-isotropic}
(or simply isotropic) if $B\vert_{Z\times Z}$ is null.
Given a self-adjoint Fredholm operator
$T$, there exists an orthogonal decomposition:
\[H=V^{-}(T)\oplus\Ker(T)\oplus V^+(T)\]
into $T$-invariant closed subspaces such that $B=\langle T\cdot,\cdot\rangle$ is
negative definite (resp., positive definite) on $V^-(T)$ (resp., on $V^+(T)$).
The \emph{index} of $B=\langle T\cdot,\cdot\rangle$ denoted by $\mathrm n_-(B)$, is the dimension of $V^{-}(T)$;
equivalently, $\mathrm n_-(B)$ is the dimension of a maximal subspace of $H$ on which $B$ is negative definite.
Observe that if $Z$ is an isotropic subspace, then $Z\cap V^-(T)=Z\cap V^+(T)=\{0\}$.
If $X\subset H$ is a subspace, we set:
\[X^{\perp_B}=\big\{y\in H:B(x,y)=0\ \forall\,x\in X\big\};\]
assume that $X$ is closed, then if $B$ and $B\vert_{X\times X}$ are nondegenerate,  $B\vert_{X^{\perp_B}\times X^{\perp_B}}$ is nondegenerate,
and $H=X\oplus X^{\perp_B}$. In this case:
\[\mathrm n_-(B)=\mathrm n_-\big(B\vert_{X\times X}\big)+\mathrm n_-\big(B\vert_{X^{\perp_B}\times X^{\perp_B}}\big).\]
\begin{lemma}\label{thm:lem1}
Let $B$ be a Fredholm bilinear form and $Z\subset H$ be a $B$-isotropic subspace such that $Z\cap\Ker(B)=\{0\}$.
Then, $\mathrm n_-(B)\ge\mathrm{dim}(Z)$.
\end{lemma}
\begin{proof}
First, we observe that we can assume $\Ker(B)=\{0\}$. Namely, should this not be the case, one can consider
the quotient $\overline H=H/\Ker(B)$, endowed with the induced Fredholm bilinear form $\overline B$, that has
the same index as $B$. If $\pi:H\to\overline H$ is the projection, since $Z\cap\Ker(B)=\{0\}$, then setting
$\overline Z=\pi(Z)$, we get a $\overline B$-isotropic subspace of $\overline H$ with the same dimension
as $Z$. This shows that it suffices to consider the case that $\Ker(B)=\{0\}$.

Assume $\Ker(B)=\{0\}$; consider the representative $T$ of $B$ and the decomposition $H=V^-(T)\oplus V^+(T)$.
If $\mathrm{dim}(Z)>\mathrm{dim}\big(V^-(T)\big)$, then it would be $Z\cap V^+(T)\ne\{0\}$, which contradicts the assumption
that $Z$ is isotropic. This concludes the proof.
\end{proof}
Let us now prove the following:
\begin{proposition}\label{thm:prop2}
 Let $X\subset H$ be a closed subspace, and let $B$ be a Fredholm
bilinear form on $H$. Assume that $X\cap\Ker(B)=\{0\}$;
then: \[\mathrm n_-(B)\ge\mathrm n_-\big(B\vert_{X\times X}\big)+\mathrm{dim}\big[\Ker\big(B\vert_{X\times X}\big)\big].\]
\end{proposition}
\begin{proof}
As in Lemma~\ref{thm:lem1}, we can assume $\Ker(B)=\{0\}$. Let $V\subset X$ be a maximal subspace on which
$B\vert_{X\times X}$ is negative definite, so that $\mathrm n_-(B\vert_{X\times X})=\mathrm{dim}(V)$, $B\vert_{V^{\perp_B}\times V^{\perp_B}}$
is nondegenerate and $H=V\oplus V^{\perp_B}$. Clearly, the kernel $\Ker\big(B\vert_{X\times X}\big)$ is an isotropic subspace of $V^{\perp_B}$,
thus, by Lemma~\ref{thm:lem1}:
\begin{multline*}\mathrm n_-(B)=\mathrm n_-\big(B\vert_{V\times V}\big)+\mathrm n_-\big(B\vert_{V^{\perp_B}\times V^{\perp_B}}\big)=
\mathrm{dim}(V)+\mathrm n_-\big(B\vert_{V^{\perp_B}\times V^{\perp_B}}\big)\\ \ge\mathrm{dim}(V)+\mathrm{dim}\big[\Ker\big(B\vert_{X\times X}\big)\big].
\end{multline*}
This concludes the proof.
\end{proof}
We will denote by $\mathcal L(H)$ the algebra of all bounded linear operators on $H$.
The \emph{Grassmannian} $\mathcal G(H)$ of all closed subspaces of $H$, endowed with the distance
$\mathrm{dist}(X,Y)=\Vert P_X-P_Y\Vert$, is a complete metric space, where $P_Z:H\to H$ denotes the orthogonal projection onto $Z\in\mathcal G(H)$
and $\Vert\cdot\Vert$ is the operator norm.

\begin{definition}\label{thm:defcontfam}
A family $\{H_s\}_{s\in[a,b]}$ of closed subspaces of $H$ is said to be a \emph{continuous
family of closed subspaces} if the map $[a,b]\ni s\mapsto H_s\in\mathcal G(H)$ is continuous.
\end{definition}
Weaker notions of continuity may also be considered (see Appendix~\ref{app:A}).

Given a projection\footnote{By a \emph{projection}, we mean an operator $P\in\mathcal L(H)$
such that $P^2=P$; by an \emph{orthogonal projection} we mean a self-adjoint projection.} $P\in\mathcal L(H)$,
we will denote by $\mathrm{Im}(P)$ the image $P(H)$, which is a closed subspace of $H$. The following lemma can be found in \cite[Lemma 4.7]{BenPicpre}.
\begin{lemma}\label{thm:lemprojprox}
Let $P,Q$ be  projections in $\mathcal L(H)$ with $\Vert P-Q\Vert<1$.
Then, the restriction $\widetilde P:\mathrm{Im}(Q)\to\mathrm{Im}(P)$ of $P$ is an isomorphism.
\end{lemma}
A self-adjoint operator $T$ in $\mathcal L(H)$ is said to be \emph{essentially positive} if it is of the form
$P+K$, where $P$ is a positive isomorphism of $H$, that is, a self-adjoint isomorphism satisfying that
$\prodot{Px,x}>0$ for every $x\in H\setminus \{0\}$, and $K$ is a compact (self-adjoint) operator on $H$.
In particular, an essentially positive operator is Fredholm.
A symmetric bilinear form $B$ will be called essentially positive if it is represented by an essentially
positive operator. Also this notion does not depend on the choice of an inner product.
An essentially positive operator has finite index; moreover, the restriction to any closed subspace
of an essentially positive form is essentially positive.
\begin{lemma}\label{lem4}
Let $B:H\times H\to \R$ be an essentially positive symmetric bilinear
form. Then $B$ has finite index.
\end{lemma}
\begin{proof}
Since $B$ is essentially positive, the
self-adjoint operator $T$ associated to $B$ can be expressed as $P+K$,
with $P$ a positive isomorphism and $K$ a compact self-adjoint operator on
$H$. Considering the equivalent scalar product $\langle\cdot,\cdot\rangle_1=\left\langle
P\cdot,\cdot\right\rangle$, the self-adjoint operator associated to $B$
can be expressed as $I+P^{-1}K$, where $I$ is the identity in $H$
and $P^{-1}K$ is compact. Note that $P^{-1}K$ is self-adjoint relatively to the inner
product $\langle\cdot,\cdot\rangle_1$.
The index of $B$ is the sum of the dimensions of
the eigenspaces of the self-adjoint compact operator $P^{-1}K$ corresponding to
its eigenvalues $\lambda<-1$; this is a finite number.
\end{proof}
The following characterization of essentially positive symmetric bilinear forms will be useful:
\begin{lemma}\label{thm:charessposforms}
Let $B$ be a bounded symmetric bilinear form on $H$. Then, $B$ is essentially positive if and only if
there exists a closed finite codimensional subspace $V$ of $H$ such that:
\begin{equation}\label{eq:infpositivo}
\inf_{\stackrel{x\in V}{\|x\|=1}}B(x,x)>0.
\end{equation}
\end{lemma}
\begin{proof}
Assume that a subspace $V$ as in the statement of the Lemma exists. Let $\widetilde P:V\to V$ be the positive isomorphism such that $B\vert_{V\times V}=\langle\widetilde P\cdot,\cdot\rangle$, and define $P:H\to H$ by
setting $P(x)=\widetilde P(x)$ for $x\in V$ and $P(x)=x$ for $x\in V^\perp$. Clearly, $P$ is a positive isomorphism
of $H$.
Moreover, the difference $B-\langle P\cdot,\cdot\rangle$ is represented by a finite rank (hence compact)
operator $K$ of $H$; namely, $K(V)\subset V^\perp$, and so $K(H)\subset V^\perp+K(V^\perp)$, which is a finite dimensional subspace of $H$. Thus, $B$ is essentially positive.

Conversely, assume that $B$ is essentially positive, and set $B=\langle(P+K)\cdot,\cdot\rangle$, where
$P$ is a positive isomorphism of $H$ and $K$ is a compact self-adjoint operator on $H$.
There exists a positive constant $c>0$ such that $\langle Px,x\rangle\ge c\Vert x\Vert^2$ for all
$x\in H$. Since $K$ is compact, there exists also a finite codimensional closed space $V$ of $H$ such that
$\vert\langle Kx,x\rangle\vert\le\frac c2\Vert x\Vert^2$ for all $x\in V$.
Namely, $V$ can be taken to be the closure of the direct sum of the
eigenspaces of $K$ corresponding to all the eigenvalues $\lambda$ of $K$ with $\vert\lambda\vert\le\frac c2$.
Now, for $x\in V$, $B(x,x)=\langle Px,x\rangle+\langle Kx,x\rangle\ge\frac c2\Vert x\Vert^2$.
This concludes the proof.
\end{proof}

Let $\mathcal{L}_{sa}(H)$ be the closed subspace of $\mathcal L(H)$ consisting of all self-adjoint operators,
and let $\mathcal B_s(H)$ denote the space of bounded symmetric bilinear forms on $H$.
Once an inner product is fixed on $H$, one has a natural identification of these two spaces
by $\mathcal{L}_{sa}(H)\ni T\mapsto\langle T\cdot,\cdot\rangle\in\mathcal B_s(H)$;
we will consider $\mathcal B_s(H)$ endowed with the induced topology.
\begin{lemma}\label{thm:4}
Let $B$ be a bounded symmetric bilinear form on $H$, let $V\subset H$ be a closed subspace
such that \eqref{eq:infpositivo} holds. If $P$ is a projection in $\mathcal L(H)$ which is
\emph{sufficiently close} to the orthogonal projection $P_V$ onto $V$ and
$\widetilde{B}\in\mathcal B_s(H)$ is close enough to $B$, then
\[\inf_{\stackrel{x\in P(V)}{\|x\|=1}}\widetilde B(x,x)>0.\]
\end{lemma}
\begin{proof}
It is a consequence of the fact that  convergence in $\mathcal B_s(H)$ means uniform convergence
on the unit sphere of $H$ and the following inequality:
\[\frac{1}{1+\|P-P_V\|}\leq\|y\|\leq \frac{1}{1-\|P-P_V\|}\]
for every $y\in V$ such that $\|P(y)\|=1$.
\end{proof}

\begin{corollary}\label{thm:4bisbis}
The set:
\[\mathcal A=\big\{(B,V)\in\mathcal B_s(H)\times\mathcal G(H):\ \text{$B\vert_{V\times V}$ is essentially positive}\big\}
\]
is open in $\mathcal B_s(H)\times\mathcal G(H)$.
The map
\begin{equation}\label{eq:map1}\mathcal A\ni (B,V)\longmapsto n_-\big(B\vert_{V\times V}\big)+
\dim\left[\Ker\left(B\vert_{V\times V}\right)\right]\in \N\end{equation} is upper semi-continuous,
and the map
\begin{equation}\label{eq:map2}\mathcal A\ni (B,V)\longmapsto n_-\big(B\vert_{V\times V}\big)\in
\N\end{equation} is lower semi-continuous.
\end{corollary}
\begin{proof}
The openness of $\mathcal A$ follows immediately from Lemma~\ref{thm:charessposforms}
and Lemma~\ref{thm:4}. Namely, if $B_{V\times V}$ is essentially positive, then there exists
a closed subspace $W\subset V$ having finite codimension in $V$ such that
$\inf\limits_{\stackrel{x\in W}{\Vert x\Vert=1}}B(x,x)>0$. Then, if $P$ is an orthogonal projection
sufficiently close to $P_V$, $P(W)$ is a finite codimensional subspace of $P(V)$, and if
$\widetilde B\in\mathcal B_s(H)$ is sufficiently close to $B$
by Lemma~\ref{thm:4} $\inf\limits_{\stackrel{x\in P(W)}{\Vert x\Vert=1}}\widetilde B(x,x)>0$.
Thus, by Lemma~\ref{thm:charessposforms}, $\widetilde B\vert_{P(W)\times P(W)}$ is essentially positive.

Given $(B,V)\in\mathcal A$, the quantity $n_-\big(B\vert_{V\times V}\big)+
\dim\left[\Ker\left(B\vert_{V\times V}\right)\right]$ is equal to the codimension in $V$ of a maximal
closed subspace $W\subset V$ on which $B$ is positive definite. Given one such $W$, an orthogonal
projection $P$ sufficiently close to $P_V$ and a symmetric bilinear form $\widetilde B$ sufficiently
close to $B$, then by Lemma~\ref{thm:4} $\widetilde B$ is positive definite on $P(W)$, and, by Lemma~\ref{thm:lemprojprox},
the codimension of $P(W)$ in $P(V)$ equals the codimension of $W$ in $V$. This proves that
\[n_-\big(B\vert_{V\times V}\big)+
\dim\left[\Ker\left(B\vert_{V\times V}\right)\right]\ge
n_-\big(\widetilde B\vert_{P(V)\times P(V)}\big)+
\dim\left[\Ker\left(\widetilde B\vert_{P(V)\times P(V)}\right)\right],\]
i.e., the upper semi-continuity of the map \eqref{eq:map1}.

Similarly, if $(B,V)\in\mathcal A$, the quantity $n_-\big(B\vert_{V\times V}\big)$ is equal to the dimension
of a maximal closed subspace $W\subset V$ on which $-B$ is positive definite. Such $W$ is necessarily finite dimensional,
hence $\inf\limits_{\stackrel{x\in W}{\Vert x\Vert=1}}-B(x,x)>0$.
Given one such $W$, an orthogonal
projection $P$ sufficiently close to $P_V$ and a symmetric bilinear form $\widetilde B$ sufficiently
close to $B$, then by Lemma~\ref{thm:4} $-\widetilde B$ is positive definite on $P(W)$, and, by Lemma~\ref{thm:lemprojprox},
the dimension of $P(W)$ is equal to the dimension of $W$. This proves that
\[n_-\big(B\vert_{V\times V}\big)\le
n_-\big(\widetilde B\vert_{P(V)\times P(V)}\big),\]
i.e., the lower semi-continuity of the map \eqref{eq:map2}.
\end{proof}
We can therefore prove the following:
\begin{proposition}[Abstract Morse Index Theorem]\label{thm:MIT}
Let $B_s:H\times H\to \R$ with $s\in[a,b]$ be a continuous family of bounded symmetric bilinear forms
and let $\{H_s\}_{s\in\left[a,b\right]}$ be a continuous family of closed subspaces of $H$
such that the restriction \[B_s:H_s\times H_s\longrightarrow \R\] is essentially positive for all $s$.
Assume that for every $s,t\in [a,b]$ with $s<t$, there exists an injective linear map
$\varphi_{\{s,t\}}:H_s\to H_t$ with closed range such that
\begin{enumerate}
\item\label{itm:2} $B_t\big(\varphi_{\{s,t\}}(V),\varphi_{\{s,t\}}(W)\big)=B_s(V,W)$ for $V,W\in H_s$;
\item\label{itm:3} $\Ker\big(B_t\vert_{H_t\times H_t}\big)\cap \varphi_{\{s,t\}}( H_s)=\{0\}$.
\end{enumerate}
Assume also that $B_a|_{H_a\times H_a}$ is non degenerate.
Then:
\begin{itemize}
\item[(a)] the map $\left[a,b\right]\ni s\mapsto\mathrm n_-\big(B_s\vert_{H_s\times H_s}\big)\in\N$ is nondecreasing;
\item[(b)] the set of instants $s\in\left]a,b\right[$ such that $\Ker\big(B_s\vert_{H_s\times H_s}\big)\ne\{0\}$ is finite;
\item[(c)] $\mathrm n_-\big(B_b\vert_{H_b\times H_b}\big)=\mathrm n_-\left( B_a\vert_{H_a\times H_a}\right)+
\sum\limits_{s\in\left]a,b\right[}\mathrm{dim}\big[\Ker\big(B_s\vert_{H_s\times H_s}\big)\big].$
\end{itemize}
\end{proposition}
\begin{proof}
Part (a) is obvious,   since by \eqref{itm:2}, the restriction of $B$ to $\varphi_{\{s,t\}}(H_s)\subset H_t$ has the same index than the restriction of $B$ to $H_s$. By Lemma~\ref{lem4}, we know that $\mathrm n_-\big(B_s\vert_{H_s\times H_s}\big)$ is finite for all $s$.
Proposition~\ref{thm:prop2} and assumptions \eqref{itm:2} and \eqref{itm:3} imply that if $t>s$
\[\mathrm n_-(B_t|_{H_t\times H_t})\geq \mathrm n_-(B_s|_{H_s\times H_s})+\dim\big[\Ker\big(B_s|_{H_s\times H_s}\big)\big].\]
A repeated use of this inequality
shows that if there existed an infinite number of instants $s\in\left]a,b\right[$
at which $B_s\vert_{H_s\times H_s}$ degenerates, then $\mathrm n_-(B\vert_{H_b\times H_b})$ would be infinite.
This is absurd, and proves (b). Corollary~\ref{thm:4bisbis} says that if there is no $s\in[c,d]$ such that $B_s\vert_{H_s\times H_s}$ degenerates,
then $\mathrm n_-\big(B_s\vert_{H_s\times H_s}\big)$ is constant on $[c,d]$; namely,
if there is no $s\in[c,d]$ such that $B_s\vert_{H_s\times H_s}$ degenerates,
then the function $s\mapsto\mathrm n_-\big(B_s\vert_{H_s\times H_s}\big)$ is both lower and upper semi-continuous
on $[c,d]$, i.e., continuous and therefore constant.
Using \eqref{itm:3} and Proposition~\ref{thm:prop2}, the jumps of the map $\mathrm n_-\big(B_s\vert_{H_s\times H_s}\big)$ at a degeneracy instant
are at least equal to the dimension of $\Ker\big(B_s\vert_{H_s\times H_s}\big)$.
On the other hand, Corollary~\ref{thm:4bisbis} says that the value of this jump is at most equal to the dimension
of $\Ker\big(B_s\vert_{H_s\times H_s}\big)$, from which the equality in (c) follows.
\end{proof}
\begin{remark}
The reader will find several analogies between the result of
Proposition~\ref{thm:MIT} and several other abstract Morse index theorems appearing in the literature,
most notably, \cite[Theorem 1.11]{Uhl73} (see also \cite{FrTh83,FrTh90}).
All these results originated from a celebrated index theorem due to Smale \cite{Sma65}
which holds for a strongly elliptic self-adjoint differential operator $L$ of even order
defined on the sections of a Riemannian vector bundle $E$ over a compact manifold with boundary $M$.
In order to obtain Smale's result, one considers the following setup:
$H$ is (a suitable closure of) the space $C^\infty(E)$ of smooth sections of $E$ vanishing on $\partial M$,
$B$ is the  bilinear form $B(u,v)=\int_M\langle Lu,v\rangle\,\mathrm dM$, and $H_s$ is the space of sections
of $E\vert_{M_s}$ vanishing on $\partial M_s$, corresponding to a smooth deformation of $M$ by a filtration
of compact submanifolds $M_s\subset M$, $s\in[a,b]$.
The strong ellipticity assumption gives that $B$ is essentially positive.
The assumption that $L$ has the \emph{uniqueness property} for the Cauchy problem, i.e., that if $u\in C^\infty(E)$
satisfies $Lu=0$ and $u$ vanishes on a nonempty open subset implies $u\equiv0$, gives assumption~\eqref{itm:3}
in Proposition~\ref{thm:MIT}. In this setup, the family $H_s$ is \emph{not} continuous in the sense
of Definition~\ref{thm:defcontfam} (see Appendix~\ref{app:A}), but only in a weaker sense.
Nonetheless, an index theorem is proved in this context using the fact that the family
$H_s$ is \emph{increasing}, i.e., $H_s\subset H_t$ when $s\le t$, in which case it suffices to require
that the family of orthogonal projections onto $H_s$ is continuous relatively to the strong
operator topology. This is the basic idea in the results of \cite{FrTh83, FrTh90, Uhl73}.
In the present paper we will consider a situation where the weak continuity of a given increasing family
of closed subspaces may fail, and  \cite[Theorem 1.11]{Uhl73} does not apply.
\medskip

\end{remark}
\end{section}

\begin{section}{Pseudo focal points and Morse-Sturm systems}
\label{sec:diffeq}
\subsection{Stationary Lorentzian manifolds and geodesics}\label{stationary}
Let $(M,g)$ be a stationary Lorentzian manifold,  $\nabla$ the associated Levi-Civita connection,
$\mathcal P$ a smooth submanifold of $M$
and  $\mathcal Y$ a timelike Killing field on $M$ (see \cite{BeErEa96,HaEl73,One83} for details).
Given a geodesic $\gamma:[0,1]\to M$,
the equation $g(\nabla_{\dot\gamma}\dot\gamma,\mathcal Y)=0$ integrates as $g(\dot\gamma,\mathcal Y)=C_\gamma$,
where $C_\gamma$ is a real constant. In \cite{GiaPic99}, it was proposed the space of $H^1$-curves
$\mathcal N_{p,q}M$ joining $p$ and $q$ in $M$ and satisfying the condition $g(\dot\gamma,\mathcal Y)=C_\gamma$
almost everywhere to study the energy functional in a stationary Lorentzian manifold. This can be generalized
as in \cite{GiaMasPic01} to the situation in that the curves depart not from a point, but from an orthogonal
initial submanifold $\mathcal P$. Let $S^{\mathcal P}_{\dot\gamma(0)}$ denote the second fundamental form
of $\mathcal P$ at the orthogonal direction $\dot\gamma(0)$. Recall that $S^{\mathcal P}_{\dot\gamma(0)}$ is the symmetric
bilinear form on $T_{\gamma(0)}\mathcal P$ defined by $S^{\mathcal P}_{\dot\gamma(0)}(v,w)=g\big(\dot\gamma(0),\nabla_vW\big)$,
where $v,w\in T_{\gamma(0)}\mathcal P$ and $W$ is any extension of $w$ to a local vector field
along $\mathcal P$.

It is a convenient assumption that $\mathcal P$ be \emph{nondegenerate} at $\gamma(0)$, i.e.,
that the restriction of the Lorentzian metric tensor $g$ to $T_{\gamma(0)}\mathcal P$ be nondegenerate.
This assumption has two basic consequences:
\begin{itemize}
\item[(a)] there are no $\mathcal P$-focal points on a sufficiently short initial portion of $\gamma$;
\smallskip

\item[(b)] $S^{\mathcal P}_{\dot\gamma(0)}$ can be written in terms of the \emph{shape operator} of $\mathcal P$,
which is a $g$-symmetric linear endomorphism, also denoted by $S^{\mathcal P}_{\dot\gamma(0)}$, defined as the linear
operator associated to the second fundamental form $S^{\mathcal P}_{\dot\gamma(0)}$ relatively to the restriction
of $g$ to $T_{\gamma(0)}\mathcal P$.
\end{itemize}
The subset $\mathcal N_{\{\mathcal P,q\}}M$ is a submanifold of the manifold $\Omega_{\{\mathcal P,q\}}M$ consisting of all $H^1$-curves from $\mathcal P$ to $q$ in $M$ satisfying  $g(\dot\gamma,\mathcal Y)=C_\gamma$. It is not difficult to show that the tangent space to $\mathcal N_{\{\mathcal P,q\}}M$ is given by the  $H^1$-vector fields  $V$ along $\gamma$ with $V(0)\in T_{\gamma(0)}\mathcal P$, $V(1)=0$ and
\begin{equation}\label{cv}
g(V',\mathcal Y)-g(V,\mathcal Y')=C_V
\end{equation}
a. e. on $[0,1]$ for any constant $C_V$ (in the following we will use the upper index $'$ to denote covariant differentiation along $\gamma$ or derivation depending on the context). Moreover, if we consider the energy functional
\[E(\gamma)=\int^1_0 g(\dot\gamma,\dot\gamma)\,\de s,\]
restricted to $\mathcal N_{\{\mathcal P,q\}}M$, its critical points are the geodesics from $\mathcal P$ to $q$ that depart orthogonally from $\mathcal P$.
Along this section we are going to consider the subspace of $T_\gamma\big(\mathcal N_{\{\mathcal P,q\}}M\big)$ putting $C_V=0$. The idea is to restrict the tangent of $\mathcal N_{\{\mathcal P,q\}}M$ to the tangent of the
subset of curves having the same constant $C_\gamma$. We observe that this subset may fail to be a submanifold of $\mathcal N_{\{\mathcal P,q\}}M$ and when it is, the critical points of the energy functional restricted to it may not be geodesics. Anyway, it will be of a great help to study the index form, which can be written as
\begin{equation}\label{index}
I_{\{\gamma,\mathcal P\}}(V,W)=\int_0^1 \big[g(V', W')+g\big(R(\dot\gamma,V)\dot\gamma,W\big)\big]\,\de t-g(S^{\mathcal P}_{\dot\gamma(0)}[V(0)],W(0)),
\end{equation}
where $R$ is the curvature tensor of $M$ chosen with the sign convention $R(X,Y)=[\nabla_X,\nabla_Y]-\nabla_{[X,Y]}$.
 Recall that a {\it Jacobi field} along $\gamma$ (see \cite{One83}) is a vector field $J$ along $\gamma$ satisfying the Jacobi equation
\[
J''=R(\dot\gamma,J)\dot\gamma;
\]
then using that the restriction of $\mathcal Y$ to $\gamma$ is a Jacobi field, it is easy to prove that $J$ satisfies Eq. \eqref{cv}. We say that $t_0\in\left]0,1\right]$ is a {\it focal instant} of the geodesic $\gamma$ with respect to $\mathcal{P}$, if there exists a non null Jacobi field $J$ satisfying $J(0)\in T_{\gamma(0)}\mathcal P$, $J'(0)+S^{\mathcal P}_{\dot\gamma(0)}[J(0)]\in \left(T_{\gamma(0)}\mathcal P\right)^\bot$,  and $J(t_0)=0$.
\subsection{Morse-Sturm systems and Jacobi fields.
}\label{parallelfocal}
The results we are going to obtain hold in the more general context of \emph{Morse-Sturm systems}, i.e.,
differential systems of the form:
\begin{equation}\label{morsesturm}
V''(t)- R(t)[V(t)]=0,
\end{equation}
where $V\in H^2([0,1];\R^n)$ and $R:[0,1]\to \mathcal{L}(\R^n)$ is a continuous map
for every $t\in [0,1]$ taking values in the space of all endomorphisms of $\R^n$ that are symmetric
relatively to a given nondegenerate symmetric bilinear form $g$ on $\R^n$.
To obtain a Morse-Sturm system from the geometrical setup, it is enough to
consider a parallel orthonormal frame along the geodesic $\gamma$, so that the Jacobi equation of the
geodesics becomes a Morse-Sturm system in $\R^n$. We will need some additional data. Let
 $g$ be a bilinear form with index 1 in $\R^n\times\R^n$
 (that in the stationary context represents the Lorentzian metric).   For every $t\in [0,1]$ we ask  $R(t)$  to be a $g$-symmetric linear map, that is, $g(R(t)[x],y)=g(x,R(t)[y])$ for every $x,y\in\R^n$. Let $Y$ be a  map $Y:[0,1]\to \R^n$ such that $g(Y(t),Y(t))<0$ and
\begin{equation*}
Y''(t)=R(t)[Y(t)],
\end{equation*}
let $P$ be a $g$-nondegenerate subspace of $\R^n$ ($P$ represents the tangent space
$T_{\gamma(0)}\mathcal P$), and  $S:P\to P$  a $g$-symmetric linear map
(that represents the shape operator $S^{\mathcal P}_{\dot\gamma(0)}$ of $\mathcal P$ at $\gamma(0)$ in the normal direction
$\dot\gamma(0)$).
We observe that the symbol $\bot$ will denote the orthogonal subspace with respect to $g$.
The initial conditions of the Morse-Sturm system \eqref{morsesturm} are given by
\begin{equation}\label{cauchy}
V(0)\in P\quad\text{and}\quad V'(0)+S[V(0)]\in P^\bot,
\end{equation}
and the associated index form of the problem is defined as
\begin{equation}\label{indexform}
I(V,W)=\int_0^1 \big[g(V', W')+g\big(R(t)[V],W\big)\big]\,\de t-g\big(S[V(0)],W(0)\big).
\end{equation}
Summing up, we will assume the initial data $(g,R,Y,P,S)$ defined above, we will refer to the solutions of \eqref{morsesturm} as {\it Jacobi fields}, and we will say that  $t_0\in\left]0,1\right]$ is a {\it focal instant} of the given data if there exists a non null Jacobi field satisfying the initial data \eqref{cauchy} and such that $J(t_0)=0$. It is easy to see that a Jacobi field $V$ satisfies
\begin{equation}\label{cv2}
g(V',Y)-g(V,Y')=C_V.
\end{equation}
\subsection{Admissible and singular Jacobi fields.}\label{admisingu}
In order to establish the results we aim to, we will need some additional properties of the Jacobi field $Y$.
In particular, the following definitions will be useful.
\begin{definition}\label{thm:defadmissible}
We say that $Y$ is {\it admissible} if $Y(0)$ and $Y'(0)$ are linearly independent
and \emph{singular} when $Y(s)$ and $Y'(s)$ are linearly dependent for every $s\in [0,1]$.
\end{definition}
If we denote
\begin{equation}\label{mdey}
m(Y)(s)=\left(\frac{Y(s)}{g\big(Y(s),Y(s)\big)}\right)'+\frac{Y'(s)}{g(Y(s),Y(s))},
\end{equation}
then $Y$ is admissible iff $m(Y)(0)\not=0$, and singular iff $m(Y)(s)=0$ for every $s\in[0,1]$.
This comes easily from the fact that as $g(Y,Y)<0$, $m(Y)(s)=0$ is equivalent to
$Y'(s)g(Y(s),Y(s))-Y(s)g(Y(s),Y'(s))=0$, and the last equality is  equivalent to $Y(s)$ and $Y'(s)$
being linearly dependent whenever $g(Y(s),Y(s))\not=0$.

We are especially interested in the case where the data comes from a geometrical setup. In fact, the initial data can be obtained from a more general context than stationary manifolds, that is, when considering a geodesic $\gamma$ in a Lorentzian manifold, a submanifold $\mathcal{P}$ orthogonal to $\gamma$ through $\gamma(0)$ and a timelike Jacobi field along $\gamma$. In this case the notion of {\it admissible} and {\it singular} Jacobi fields can be brought in the obvious way.

Even if we find a timelike Jacobi field $\mathcal Y$ along $\gamma$, it might not be admissible or singular. To overcome this situation
we can consider the family of Jacobi fields $\widetilde{ \mathcal Y}=\mathcal Y+(a+b\,t)\dot\gamma$ for $a,b\in\R$ small enough 
and look for a Jacobi field with the required properties.
\begin{lemma}\label{perturbation}
Let $(M,g)$ be a Lorentzian manifold, $\gamma$ a geodesic in $M$ and  $\mathcal Y$ a timelike Jacobi field along $\gamma$. Then:
\begin{itemize}
\item[({\it i})] If $\gamma$ is timelike,   $\dot \gamma$ is a singular Jacobi field along $\gamma$.

\item[({\it ii})] Consider a Jacobi field $\bar{\mathcal Y}(a,b)=\mathcal Y+(a+b\,t)\dot\gamma$ for some $a,b\in\R$, such that $\bar{\mathcal Y}(a,b)$ is timelike (for example when $a,b\in\R$ are small enough). If $\gamma$ is lightlike or spacelike, then there exist $a$ and $b$ in $\R$ such that $\bar {\mathcal Y}=\mathcal Y +(a+b t)\dot\gamma$ is admissible.
\item[({\it iii})]If $\gamma$ is spacelike, then the Jacobi field $\bar{\mathcal Y}=\mathcal Y-g(\mathcal Y,\dot\gamma)E_\gamma^{-1}\dot\gamma$, where $E_\gamma=g(\dot\gamma,\dot\gamma)$, is timelike and orthogonal to the geodesic $\gamma$.
\end{itemize}
\end{lemma}
\begin{proof}
The first assertion is obvious and $({\it iii})$ can be shown by a straightforward computation. Let us prove $({\it ii})$. Assume that $\mathcal Y$ is not admissible, that is, there exists $\alpha$ such that $\mathcal Y'(0)=\alpha \mathcal Y(0)$. Then $\bar {\mathcal Y}'(0)=\mathcal Y'(0)+b\dot\gamma(0)$, so that $\bar{\mathcal Y}$ is not admissible when there exists $\beta$ satisfying $\mathcal Y'(0)+b\dot\gamma(0)=\beta (\mathcal Y(0)+a\dot\gamma(0))$. This implies that
$(\alpha-\beta)\mathcal Y(0)=(\beta a-b)\dot\gamma(0).$
As $\mathcal Y(0)$ is timelike and $\dot\gamma(0)$ does not, it follows that $\beta=\alpha$ and $\beta a=b$, but as $\alpha$ is fixed, we can choose $a$ and $b$ small enough such that $\beta a-b=\alpha a-b\not=0$ and $\bar{\mathcal Y}$ is timelike.
\end{proof}
In the following lemma we are going to give a geometric characterization of singularity for a vector field related to $\gamma$.
\begin{lemma}\label{frame}
Let $(M,g)$ be a Lorentzian manifold and $\mathcal Y$ a timelike Jacobi field along a spacelike geodesic $\gamma:[0,1]\to M$. If we assume that $\mathcal Y$ is orthogonal to $\dot\gamma$, which is not restrictive by Lemma~\ref{perturbation}, then $\mathcal Y$ is singular if and only if there exists a $(n-1)$-tuple of parallel orthogonal vector fields $F=\{E_1,\dots,E_{n-1}\}$ along $\gamma$ such that $\{Y(s)\}^\bot=\spa\{E_1(s),\dots,E_{n-1}(s)\}$ for every $s\in[0,1]$.
\end{lemma}
\begin{proof}
If $\mathcal Y$ is singular, it is easy to see that there exists $\alpha:[0,1]\to \R\setminus\{0\}$ such that the vector field $t\to \alpha(t)\mathcal Y(t)$ is parallel. Indeed, if $\mathcal Y'(t)=\beta(t) \mathcal Y(t)$, we can choose $\alpha(t)=e^{-\int_0^t\beta(s)\de s}$. Then considering an orthonormal frame of $\mathcal Y(0)^{\bot}$ and making the parallel transport along $\gamma$ we obtain the family $F$. The other side can be shown as follows. We know that $g(\mathcal Y(t),E_i(t))=0$ for every $t\in[0,1]$ and $E_i$ are parallel along $\gamma$, so that $g(\mathcal Y'(t),E_i(t))=0$ for $i=1,\ldots,n-1$ and $\mathcal Y'(t)$ has to be linearly dependent to $\mathcal Y(t)$ for every $t\in[0,1]$.
\end{proof}
\begin{remark}\label{tothyp}
Lemma~\ref{frame} gives a relation between the geodesics admitting a singular Jacobi field and those that are contained in a totally geodesic hypersurface. It is clear that when the geodesic is contained in a totally geodesic spacelike hypersurface and there exists a timelike Jacobi field orthogonal to the hypersurface, then there exists a frame as in Lemma~\ref{frame} and a singular Jacobi field $\mathcal Y$ along $\gamma$.
\end{remark}

\subsection{Functional analytical setup}\label{sub:functanlsetup} In this subsection we will introduce several $L^2$-spaces and will state some density results, that will be used in the next subsection to compute the kernel of a restriction of the index form.
Let us consider the Hilbert space $L^2([a,b];\R^n)$ of Lebesgue integrable functions from $[a,b]$ to $\R^n$ and the Sobolev space $H^1_0([a,b];\R^n)$ of all absolutely continuous maps from $[a,b]$ to $\R^n$ vanishing in the endpoints and having derivatives in $L^2([a,b];\R^n)$. Analogously, $H^2([a,b];\R^n)$ is the space of $C^1$ maps, with an absolutely continuous first derivative and whose second derivative is in $L^2([a,b];\R^n)$. Moreover, $H^1_P([a,b];\R^n)$ consists of the functions $V\in H^1([a,b];\R^n)$ such that $V(a)\in P$ and $V(b)=0$, being $P$ a subspace of $\R^n$.

Using $Y$, we can define a smooth family of positive definite inner products $g_t^{(r)}$ on $\R^n$ as
\begin{equation}\label{riemannianmetric}
g_t^{(r)}(V,W)=g(V,W)-2\frac{g(V,Y(t))\cdot g(W,Y(t))}{g(Y(t),Y(t))}.
\end{equation}
We observe that there is a smooth family $A:[0,1]\to \mathcal{L}(\R^n)$ of $g_t^{(r)}$-symmetric operators such that
$$g(V,W)=g_t^{(r)}(A(t)[V],W)$$
for every $V,W\in\R^n.$
We also define the following inner product in the Hilbert space $L^2([0,\sigma];\R^n)$:
\begin{equation}\label{innerproduct}
{\mathcal R}_\sigma(V,V)=\int^\sigma_0g_t^{(r)}(V,V)\de t.
\end{equation}
We will now introduce two subspaces of $L^2([0,\sigma];\R^n)$, that reproduce the $L^2$-version of the space $T_\gamma\mathcal N_{\{p,q\}}M$
in the geometrical setup and a one-codimensional subspace obtained by setting $C_V=0$:
\begin{multline*}
\K(\sigma)=\Big\{V\in L^2([0,\sigma];\R^n)\, :\, \\g\big(V(t),Y(t)\big)=-2\frac{t}{\sigma}\int_0^\sigma g(V,Y')\,\de s+2\int_0^t g(V,Y')\,\de s\,\, \text{\ a. e.\ on $[0,\sigma]$}\Big\}
\end{multline*}
and
\begin{multline*}
\K_0(\sigma)=\Big\{V\in L^2([0,\sigma];\R^n)\, :g\big(V(t),Y(t)\big)=2\int_0^t g(V,Y')\,\de s\,\,
\text{\ a. e.\ on $[0,\sigma]$}\\
\text{and}\,\, \int_0^\sigma g(V,Y')\,\de s=0\Big\}.
\end{multline*}
The spaces $\K(\sigma)$ and $\K_0(\sigma)$ can also be described as follows:
\begin{multline*}\K(\sigma)=\Big\{V\in L^2([0,\sigma];\R^n)\, :\, g(V,Y)\in H^1_0([0,\sigma];\R)  \\ \text{and $\exists\ C_V\in\R$ such that}\,\, \frac{\de}{\de t}g(V,Y)=C_V+2g(V,Y')\Big\},
\end{multline*}
\begin{multline*}
\K_0(\sigma)=\Big\{V\in L^2\big([0,\sigma];\R^n\big)\, :\, g(V,Y)\in H^1_0\big([0,\sigma];\R\big)  \,\,\text{and}\\ \frac{\de}{\de t}g(V,Y)=2g(V,Y')\Big\}.
\end{multline*}
Moreover, it is easy to see that $\K(\sigma)$ and $\K_0(\sigma)$ are closed subpaces of $L^2([0,\sigma];\R^n)$. In order to simplify notations, we will omit the argument $\sigma$ when unnecessary.
We want to show that $\K_0(\sigma)\cap H^1_0\big([0,\sigma];\R^n\big)$ is dense in $\K_0(\sigma)$. The proof of this fact is based in the following abstract result.
\begin{lemma}[Density criteria]\label{thm:densitycriterion}
Let $H$ be a Hilbert space and let $R\subset H$ be a dense linear subspace.
\begin{itemize}
\item If $H_1\subset H$ is a closed subspace such that there exists a projection $P$ (not necessarily self-adjoint) from $H$ onto $H_1$ with $P(R)\subset R$,
then, $R\cap H_1$ is dense in $H_1$.
\item If $H_1\subset H$ is a closed subspace with finite codimension in $H$, then $R\cap H_1$ is dense in $H_1$.
\end{itemize}
\end{lemma}
\begin{proof}
Fix $x\in H_1$ and let $r_n\in R$ be a sequence with $\lim r_n=x$.
Then, $P(r_n)\in R\cap H_1$, because $P(R)\subset R$, and $\lim P(r_n)=P(x)=x$, which proves that $R\cap H_1$ is dense in $H_1$.
For the second statement, first note that $H_1+R$ is closed, because\footnote{Recall that any subspace that contains a closed finite codimensional subspace
is also closed.} it contains $H_1$, and dense, because it contains $R$; thus $H_1+R=H$.
We can therefore find a finite dimensional complement $H_2$ to $H_1$ such that $H_2\subset R$. Then,  the projection $P$
onto the first factor $H=H_1\oplus H_2\to H_1$ satisfies $P(R)\subset R$, and by the first density criterion $R\cap H_1$ is dense in $H_1$.
\end{proof}
\begin{proposition}\label{prop:densitainK}
$\K_0(\sigma)\cap H^1_0\big([0,\sigma];\R^n\big)$ is dense in $\K_0(\sigma)$.
\end{proposition}
\begin{proof}
By Corollary 3.2 in \cite{MasPic03} we know that $\K(\sigma)\cap H^1_0\big([0,\sigma];\R^n\big)$ is dense in $\K(\sigma)$. Then, the thesis is obtained easily from the second density criterion in Lemma~\ref{thm:densitycriterion},
applied to the Hilbert space $H=\K(\sigma)$, the dense linear subspace $R=\K(\sigma)\cap H^1_0\big([0,\sigma];\R^n\big)$,
and the closed subspace $H_1=\K_0(\sigma)$, that has codimension $1$ in $H$ (it is the kernel of the bounded linear functional
$\K(\sigma)\ni V\mapsto C_V\in\R$).
\end{proof}
\subsection{The kernel of the restricted index form} As a previous result to the computation of the kernel of the restricted index form in Proposition \ref{thm:prop2.6}
we need a description of the orthogonal space of $\K_0$ with respect to the Hilbert structure given by \eqref{innerproduct}, that we denote $\K_0^\bot$. First, we observe that $\K_0$ can be described as intersection of kernels of  bounded linear operators between Hilbert spaces. Indeed, we consider the operators
$L^2([0,\sigma];\R^n)\ni V\to T_1(V)(t)=g\big(V(t),Y(t)\big)-2\int^t_0 g(V,Y')\de s\in L^2([0,\sigma];\R)$ and $L^2([0,\sigma];\R^n)\ni V\to T_2(V)(t)=\int^\sigma_0 g(V,Y')\de s\in L^2([0,\sigma];\R) $
then $\K_0=T_1^{-1}(0)\cap T_2^{-1}(0)$.
Recall now the following abstract result in Banach spaces that can be found in \cite[Lemma 3.4]{MasPic03}.
\begin{lemma}\label{abstractbanach}
Let $X$ and $Y$ be Banach spaces and $T:X\to Y$ be a bounded linear operator with closed image in $Y$. Then, $\mathrm{Im}(T^*)=(\mathrm{Ker}(T))^0$; where $(\mathrm{Ker}(T))^0$ is the annihilator of $\mathrm{Ker}(T)$ in $X^*$. In particular, $\mathrm{Im}(T^*)$ is closed in $X$.
\end{lemma}
Therefore, we have to compute the adjoint operators
\[T_1^*,T_2^*: L^2([0,\sigma];\R)\to L^2([0,\sigma];\R^n)\]
with respect to the usual product in $L^2([0,\sigma];\R)$ and the product defined in \eqref{innerproduct} in $L^2([0,\sigma];\R^n)$. It is easily seen that $T_1^*$ and $T_2^*$ can be expressed as
\begin{equation}
T_1^*(\phi)(t)=\phi(t)\cdot (A(t)[Y(t)])-2(A(t)[Y'(t)])\cdot\int^\sigma_t\phi(s)\de s.
\end{equation}
and
\begin{equation*}
T_2^*(\phi)(t)=\int_0^\sigma \phi(s)\,\de s\, A[Y'(t)].
\end{equation*}
\begin{lemma}\label{tclosed}
The image ${\rm Im}(T_1)$ is closed in $L^2([0,\sigma];\R)$.
\end{lemma}
\begin{proof}
Consider the map $\widetilde{T}: L^2([0,\sigma];\R)\to L^2([0,\sigma];\R)$ defined as $\widetilde{T}(\mu)=T_1(\mu\cdot Y)$. By the definition of $T_1$, the operator $\widetilde{T}$ is the sum of the isomorphism $\mu\to \mu\cdot g(Y,Y)$ and a compact operator on $L^2([0,\sigma];\R)$, so that by the Fredholm's alternative we conclude that ${\rm Im}(\widetilde{T})$ is closed and has finite codimension in $L^2([0,\sigma];\R)$. Finally, ${\rm Im}(\widetilde{T})\subset {\rm Im}(T_1)$ implies that ${\rm Im}(T_1)$ is closed, because it contains a closed subspace with finite codimension.
\end{proof}
Moreover, as the image of $T_2$ is the subset of constant functions, then it is also closed.
Hence, by Lemma~\ref{abstractbanach} and Lemma~\ref{tclosed}, determining $\K_0(\sigma)^\bot$ is equivalent to obtaining a description of $\mathrm{Im}(T_1^*)+\mathrm{Im}(T_2^*)$. With this in mind, we observe that given a function $\phi\in L^2([0,\sigma];\R)$ there exists a unique  $h_\phi\in H^1([0,\sigma];\R)$ such that $h_\phi(\sigma)=0$ and $h'_\phi=\phi$. The following corollary follows straightforward.
\begin{corollary}\label{orthospace}
The orthogonal space $\K_0(\sigma)^\bot$ in $L^2([0,\sigma];\R^n)$ is
\begin{equation}
\K_0(\sigma)^\bot=\big\{h'\cdot A[Y]+2 h\cdot A[Y']\, : h\in H^1([0,\sigma];\R)\big\}.
\end{equation}
\end{corollary}
Let us consider the following symmetric bilinear form on $H^1_P([0,\sigma];\R^n)$ given by
\begin{equation}\label{isigma}
I_\sigma(V,W)=\int_0^\sigma\Big[g(V',W')+g(R[V],W)\Big]\de t-g\big(S[V(0)],W(0)\big)
\end{equation}
and let us denote
\[{\mathcal W}_P(\sigma)=\{V\in H_P^1([0,\sigma];\R^n):g(V', Y)-g(V,Y')=0\,\,
\text{\ a. e.\ on $[0,\sigma]$}\}.\]
In order to describe the kernel of $I_\sigma$ we will introduce the following generalization of Jacobi fields.
\begin{definition}\label{pseudo}
We will say that $V\in H^2([0,\sigma];\R^n)$ is a {\it $Y$-pseudo Jacobi field} if there exists $\lambda\in\R$ such that
\[V''-
R(t)[V]=\lambda\, m(Y),
\]
(see \eqref{mdey}) and $g(V', Y)-g(V,Y')=0$. Moreover, we say that $V$ is a {\it $(P,Y)$-pseudo Jacobi field} when in addition it holds the initial conditions
\[V(0)\in P\quad\text{and}\quad V'(0)+\lambda \frac{Y(0)}{g(Y(0),Y(0))}+S[V(0)]\in P^\bot,\]
(when $Y$ is singular, we take $\lambda=0$).
\end{definition}
When the choice of $Y$ is clear by the context we will say just pseudo Jacobi or $P$-pseudo Jacobi fields.
\begin{proposition}\label{thm:prop2.6}
A vector $V_\sigma\in{\mathcal W}_P(\sigma)$ belongs to the kernel of the restriction of $I_\sigma$ to ${\mathcal W}_P(\sigma)\times {\mathcal W}_P(\sigma)$  if and only if $V_\sigma$ is a $P$-pseudo Jacobi field.
\end{proposition}
\begin{proof}
If $V_\sigma\in{\mathcal W}_P(\sigma)$ belongs to the kernel of $I_\sigma$ restricted to ${\mathcal W}_P(\sigma)\times {\mathcal W}_P(\sigma)$, then using a standard boot-strap argument
 one proves that $V_\sigma$ is differentiable. By applying integration by parts we obtain that
\begin{equation*}
I_\sigma(V_\sigma,W)=\int_0^\sigma g(-V''_\sigma+R[V_\sigma],W)\de s
\end{equation*}
for every $W\in H_0^1([0,\sigma];\R^n)\cap {\mathcal W}_P(\sigma)=H_0^1([0,\sigma];\R^n)\cap {\mathcal K}_0(\sigma)$. In particular, it follows that
\begin{equation}\label{orthojacobi}
-A[V''_\sigma]+A[R[V_\sigma]]\in \big(\K_0(\sigma)\cap H^1_0([0,\sigma],\R^n)\big)^\bot=\K_0(\sigma)^\bot,
\end{equation}
where $\bot$ is taken with respect to the scalar product \eqref{innerproduct}. This is because $\K_0(\sigma)\cap H^1_0([0,\sigma];\R^n)$ is dense in $\K_0(\sigma)$ (Proposition \ref{prop:densitainK}).
  From Eq.~\eqref{orthojacobi} and Corollary~\ref{orthospace} we deduce that there exists a function $h\in H^1([0,\sigma];\R^n)$ such that
\begin{equation}\label{jacobih}
-V''_\sigma+R[V_\sigma]=h'\cdot Y+2h\cdot Y'.
\end{equation}
Then multiplying by $Y$ with the $g$-scalar product we get
\begin{equation*}
g(-V''_\sigma,Y)+g(R[V_\sigma],Y)=\left(h\cdot g(Y,Y)\right)'.
\end{equation*}
Observing that  $V_\sigma\in {\mathcal W}_P(\sigma)$ satisfies $g(V''_\sigma,Y)=g(V_\sigma,Y'')$, $R$ is $g$-symmetric and $Y$ is a Jacobi field, we deduce that $(h\cdot g(Y,Y))'=0$. This implies that $h=\mu\, g(Y,Y)$ for some real constant $\mu$. Substituting in \eqref{jacobih} we obtain that $V_\sigma$ is a pseudo Jacobi field. Applying again integration by parts to $I_\sigma(V_\sigma,W)$, now with $W\in {\mathcal W}_P(\sigma)$, and using that $V_\sigma$ is a pseudo Jacobi field, we obtain that
$$I_\sigma(V_\sigma,W)=-g(V'_\sigma(0)+\lambda \frac{Y(0)}{g(Y(0),Y(0))}+S[V_\sigma(0)],W(0)).$$
As there exists a vector field $W\in{\mathcal W}_P(\sigma)$ such that $W(0)=U$ for every $U\in P$, we deduce that $V'_\sigma(0)+\lambda\, Y(0)/g(Y(0),Y(0))+S[V_\sigma(0)]\in P^\bot$ and therefore $V_\sigma$ is a $P$-pseudo Jacobi field.
 The other way is straightforward.
\end{proof}
\end{section}
\begin{section}{The Morse Index Theorem in stationary spacetimes}\label{sec:morseindextheorem}
\subsection{Smooth family of Hilbert spaces}
We have now enough information to prove a Morse index theorem for the index form in \eqref{indexform} in a suitable restriction
by applying the abstract theorem stated in Proposition \ref{thm:MIT}.  We will proceed studying the evolution of the index of $I_\sigma$ when $\sigma$ goes to $1$.
As we mentioned in the introduction, we cannot assure any kind of continuity of the path $\sigma\to{\mathcal W}_P(\sigma) $, so that we will consider another one obtained as a reparametrization in the interval $[0,1]$. As a matter of fact, we have
\begin{equation}\label{mapphi}
\Phi_\sigma:\Hi_P(\sigma)\longrightarrow{\mathcal W}_P(\sigma),
\end{equation}
where
\begin{equation}\label{hsigma}
\Hi_P(\sigma)=\big\{V\in H^1_P([0,1],\R^n)\,:\, g(V'(t),Y(\sigma t))-\sigma g(V(t),Y'(\sigma t))=0\big\}
\end{equation} and $\Phi_\sigma(V)=\widetilde{V}$ is given by
$s\to \widetilde{V}(s)=V(\frac{s}{\sigma})$, which is clearly one-to-one. We observe that $\Hi_P(\sigma)$ can be extended to $\sigma=0$ putting
\begin{align*}
\Hi_P(0)&=\big\{V\in H^1_P([0,1],\R^n)\,:\, g\big(V'(t),Y(0)\big)=0\big\}\\
&=\big\{V\in H^1_P([0,1],\R^n)\,:\, g\big(V(t),Y(0)\big)=0\big\}.
\end{align*}
Analogously, we define
\begin{equation*}
\Hi^*_P(\sigma)=\big\{V\in H^1_P\big([0,1],\R^n\big)\,:\, g\big(V'(t),Y(\sigma t)\big)-\sigma g\big(V(t),Y'(\sigma t)\big)=C_V\in\R\big\},
\end{equation*}
for $\sigma\in[0,1]$.
Let us show that the family of subspaces $\Hi_P(\sigma)$ varies smoothly with $\sigma$.
In order to formalize this fact, one needs to use the differentiable structure of the
Grassmannian of all closed subspaces of a Hilbert space, see for instance reference
\cite{AbboMaj03}. In analogy with Definition~\ref{thm:defcontfam}
we give the following:
\begin{definition}
\label{thm:defC1subspaces}
Let $\mathfrak H$ be a Hilbert space, $I\subset\R$ an interval and
$\{\mathcal D^t\}_{t\in I}$ be a family of closed subspaces of
$\mathfrak H$. We say that $\{\mathcal D^t\}_{t\in I}$ is a {\em
$C^1$-family  of closed subspaces} if the map $I\ni t\mapsto\mathcal D^t\in\mathcal G(\mathfrak H)$
is of class $C^1$.
\end{definition}
It is not hard to show that  $\{\mathcal D^t\}_{t\in I}$ is a
$C^1$-family  of closed subspaces if for all $t_0\in I$ there exist
$\varepsilon>0$, a $C^1$-curve
$\alpha:\left]t_0-\varepsilon,t_0+\varepsilon\right[\cap I\mapsto
\mathcal L(\mathfrak H)$ and a closed subspace $\overline{\mathcal
D}\subset\mathfrak H$ such that $\alpha(t)$ is an isomorphism  and
$\alpha(t)(\mathcal D^t)=\overline{\mathcal D}$ for all~$t\in\left]t_0-\varepsilon,t_0+\varepsilon\right[$.
Moreover, the following criterion for the smoothness of a family of closed subspaces can be found in \cite[Lemma 2.9]{GiaMasPic01}.
\begin{proposition}
\label{thm:produce}
Let $I\subset\R$ be an interval, $\mathfrak H,\widetilde{\mathfrak H}$ be Hilbert
spaces and $F:I\mapsto\mathcal L(\mathfrak H,\widetilde{\mathfrak H})$ be a $C^1$-map
such that each $F(t)$ is surjective. Then, the family $\mathcal D^t=\mathrm{Ker}\big(F(t)\big)$
is a $C^1$-family of closed subspaces of $\mathfrak H$.\qed
\end{proposition}
\begin{proposition}\label{c1}
Assume that $Y$ is singular and $Y(0)$ is orthogonal to $P$ or that $Y(0)$ is not orthogonal to $P$. Then the family of closed subspaces $\Hi_P(\sigma)$ with $\sigma\in [0,1]$ defined in \eqref{hsigma} is a $C^1$-family of $H^1_P([0,1];\R^n)$. If  $Y(0)$ is orthogonal to $P$ and $Y$ is admissible, then \eqref{hsigma} is a $C^1$-family of $H^1_P([0,1];\R^n)$ in $\left]0,1\right]$.
\end{proposition}
\begin{proof}
Consider the map $F_\sigma:H_P^1([0,1];\R^n)\to L^2([0,1];\R)$ defined as
\[F_\sigma(V)(t)=g(V'(t),Y(\sigma t))-\sigma g(V(t),Y'(\sigma t))\]
for $\sigma\in [0,1]$. We can show that $F_\sigma:[0,1]\to \mathcal{L}(H_P^1([0,1];\R),L^2([0,1];\R))$ is $C^1$ as in \cite[Lemma 4.3]{GiaMasPic01}. Moreover, by \cite[Lemma 4.4]{GiaMasPic01} we know that given a function $h\in L^2([0,1];\R)$ there exists $\hat{f}\in H^1_0([0,1];\R)$ such that if $\hat{Y}_\sigma(u)=Y(\sigma u)$, then
\[F_\sigma(\hat{f}\cdot\hat{Y}_\sigma)=h+C\]
for a certain $C\in\R$. In order to prove that $F_\sigma$ is surjective it is enough to find for any $C\in\R$ functions $W\in H_0^1([0,1];\R^n)$ and $\hat{h}\in H^1_0([0,1];\R)$ such that
\[F_\sigma(W+\hat{h}\cdot\hat{Y}_\sigma)=C.\]
The last equation is equivalent to
\[g(W'(u),Y(\sigma u))-\sigma g(W(u),Y'(\sigma u))+\hat{h}'(u)g(Y(\sigma u),Y(\sigma u))=C,\]
so that
\[\hat{h}(u)=\int_0^u \frac{C}{g(Y(\sigma s),Y(\sigma s))}\,\de s+\int_0^u \frac{-g(W'(s),Y(\sigma s))+\sigma g(W(s),Y'(\sigma s))}{g(Y(\sigma s),Y(\sigma s))}\,\de s.\]
If we find a function $W\in H_P^1([0,1];\R^n)$ such that $\hat{h}(1)=0$ the result is proven. To this end it is enough to show that there exists $W\in H_P^1([0,1];\R^n)$ such that
\[\int_0^1 \frac{-g(W'(s),Y(\sigma s))+\sigma g(W(s),Y'(\sigma s))}{g(Y(\sigma s),Y(\sigma s))}\,\de s\not=0.\]
By applying integration by parts this is equivalent to
\begin{equation}\label{condw}
-\frac{g(W(0),Y(0))}{g(Y(0),Y(0))}+\sigma\int_0^1 g(W(s),m(Y)(\sigma s))\,\de s\not=0,
\end{equation}
where $m(Y)$ is defined in \eqref{mdey}.
When $\sigma=0$, this condition is satisfied iff $Y(0)$ is not orthogonal to $P$. For $\sigma\in\left]0,1\right]$, Eq.~\eqref{condw} is satisfied for some $W\in H^1_P([0,1]; \R^n)$ if and only if $m(Y)(\sigma s)\not=0$ for some $s\in[0,1]$ or $Y(0)$ is not orthogonal to $P$. If $Y$ is admissible then $m(Y)(\sigma s)\not=0$ for $s$ small enough. Summing up, $F_\sigma$ is surjective for $\sigma\in[0,1]$ when $Y(0)$ is not orthogonal to $P$ and for $\sigma\in \left]0,1\right]$ when $Y$ is admissible. Applying Proposition \ref{thm:produce} we conclude that $\Hi_P(\sigma)$ is a $C^1$-family in both cases. Assume now that $Y$ is singular and $Y(0)$ is orthogonal to $P$. We will see that in this case $\Hi_P(\sigma)=\Hi_P^*(\sigma)$ , so that we can apply \cite[Corollary 4.5]{GiaMasPic01} to conclude that $\Hi_P(\sigma)$ is a $C^1$-family. A function $W\in \Hi^*_P(\sigma)$ satisfies that
\[-g(W'(s),Y(\sigma s))+\sigma g(W(s),Y'(\sigma s))=C_W.\]
Dividing by $g(Y(\sigma s),Y(\sigma s))$,  integrating between $0$ and $1$ and applying integration by parts, we obtain
\[-\frac{g\big(W(0),Y(0)\big)}{g\big(Y(0),Y(0)\big)}+\int^1_0 g\big(W(s),m(Y(\sigma s))\big)\,\de s=C_W\int_0^1 \frac{\de s}{g\big(Y(\sigma s),Y(\sigma s)\big)}.\]
The left term is zero and the right term is zero iff $C_W=0$, so that we conclude that the constant $C_W$ has to be zero and therefore, $\Hi_P(\sigma)=\Hi^*_P(\sigma)$.
\end{proof}
\subsection{Morse index and nullity}
We must find the counterpart of the index form in $\Hi_P(\sigma)$. Using the map \eqref{mapphi} and the index form $I_\sigma$ given in \eqref{isigma}, we obtain  $\hat{I}_\sigma(V,W)=I_\sigma(\Phi_\sigma(V),\Phi_\sigma (W))$; more explictly,
\begin{equation}\label{isigmaenH}
\hat{I}_\sigma(V,W)=\int_0^1\left[\tfrac{1}{\sigma} g\big(V'(t),W'(t)\big)+\sigma g\big(R(\sigma t)[V(t)],W(t)\big)\right]\,\de t-g\big(S[V(0)],W(0)\big).
\end{equation}
We observe that $C_\sigma=\sigma \hat{I}_\sigma$ can be extended to $\sigma=0$ in a continuous way as
\[C_0(V,W)=\int_0^1 g\big(V'(t),W'(t)\big)\,\de t.\]
\begin{proposition}\label{finiteindex}
The Morse index  of ${\hat{I}_\sigma}|_{\Hi_P(\sigma)\times \Hi_P(\sigma)}$ is finite for all $\sigma\in \left]0,1\right]$ and the related operator  to $\hat{I}_\sigma$ is essentially positive.
\end{proposition}
\begin{proof}
The first part of the proposition follows from the second one and Lemma~\ref{lem4}.
By using Eq.~\eqref{riemannianmetric} and the identity $g(V'(t),Y(\sigma t))=\sigma g(V(t),Y'(\sigma t))$ for $V\in\Hi_P(\sigma)$, $C_\sigma=\sigma\hat{I}_\sigma$ can be expressed as
\begin{align}\label{compactness}
C_\sigma(V,W)=&\int_0^1 g_{\sigma t}^{(r)}(V'(t),W'(t))\de t
+2\sigma^2\int_0^1 \frac{g(V(t),Y'(\sigma t))g(W(t),Y'(\sigma t))}{g(Y(\sigma t),Y(\sigma t))}\de t\nonumber\\&+ \sigma^2\int_0^1 g(R(\sigma t)[V(t)],W(t))\de t-g(S[V(0)],W(0)).
\end{align}
The first term gives the identity times a constant as associated operator with respect to the scalar product given by
\begin{equation}\label{scalarproduct}
\int_0^1 g_{\sigma t}^{(r)}(V'(t),W'(t))\de t,
\end{equation}
for $V,W\in H_P^1([0,1];\R^n)$, but the positivity of the related operator does not depend on the scalar product. The other terms give a continuous operator that has to be compact because $H_P^1([0,1];\R^n)$ is compactly embedded in $C^0([0,1];\R^n)$.
\end{proof}
The case $\sigma=0$ must be considered separately, because the path $\Hi_P(\sigma)$ may not be even continuous
at that instant and when it is, we need to compute the index of $C_0$ to establish the Morse index theorem.
\begin{lemma}\label{positivedefinite}
If $Y(0)$ is orthogonal to $P$, then
the index of $\hat{I}_\sigma$ is zero if $\sigma$ is small enough. If $Y(0)$ is not orthogonal to $P$, the index of $C_0$ is zero.
\end{lemma}
\begin{proof}
We first observe that if $Y(0)$ is orthogonal to $P$, then ${\rm n}_-(g|_P)=0$, so that by \cite[Proposition 4.10
]{GiaMasPic01} we know that $\hat{I}_\sigma$ restricted to $\Hi^*_P(\sigma)\times\Hi^*_P(\sigma)$ is positive definite if $\sigma$ is small enough. As $\Hi_P(\sigma)\subset \Hi^*_P(\sigma)$, the thesis follows. Assume that $Y(0)$ is not orthogonal to $P$.  In \cite[Proposition 4.10]{GiaMasPic01} it was shown that $\Hi^*_P(0)$ can be decomposed as a direct sum $\left(\Hi^*_P(0)\right)_+\oplus\left(\Hi^*_P(0)\right)_-$, where
\begin{align*}
\left(\Hi^*_P(0)\right)_+&=\{V\in\Hi^*_P(0)\, :\, V(0)\in P_+\},\\
\left(\Hi^*_P(0)\right)_-&=\{V\, :\, [0,1]\to\R^n\,\text{affine function}\, :\, V(0)\in P_-\, , \, V(1)=0\}
\end{align*}
and $P=P_+\oplus P_-$ is a decomposition of $P$ as a direct sum of a positive and a negative space, in such a way that $C_0$ is positive definite in $\left(\Hi^*_P(0)\right)_+$ and negative definite in $\left(\Hi^*_P(0)\right)_-$. On other hand, if $V\in \Hi_P(0)$, then $V(0)\in \{Y(0)\}^\bot\cap P$; moreover as $g$ is positive definite on $\{Y(0)\}^\bot\cap P$, we can choose a decomposition $P=P_+\oplus P_-$ in such a way that $\{Y(0)\}^\bot\cap P\subset P_+$. Then $\Hi_P(0)\subset \left(\Hi^*_P(0)\right)_+$ and $C_0$  is positive definite on $\Hi_P(0)$.
\end{proof}
\begin{definition}\label{defpseudo}
An instant $t_0\in(0,1]$ is said {\it $(P,Y)$-pseudo focal} if there exists a $(P,Y)$-pseudo Jacobi field such that $V(t_0)=0$.
\end{definition}
Finally we can get a Morse index theorem of Riemannian type.
\begin{theorem}\label{Morseindex}
Assume that $Y(0)$ is not orthogonal to $P$ or that $Y$ is either admissible or singular. Then, the Morse index of $I|_{\Hi_P(1)\times \Hi_P(1)}$ coincides with the number of $(P,Y)$-pseudo focal points counted with multiplicity.
\end{theorem}
\begin{proof}
It follows from the Abstract Morse Index Theorem given in Proposition~\ref{thm:MIT} by taking
$H_s=\Hi_P(s)\subset H^1([0,1];\R^n)$ and $$B_s=C_s:H^1([0,1];\R^n)\times H^1([0,1];\R^n)\to \R$$
(see Eq. \eqref{isigmaenH} and the following paragraph)  defined for  $s\in[0,1]$ when $Y(0)$
is not orthogonal to $P$ and on $[\varepsilon,1]$ with $\varepsilon>0$ small enough such that
$\hat{I}_s$ is positive definite in $\Hi_P(s)\times \Hi_P(s)$ when $s\in(0,\varepsilon]$ in the
other case. It is easy to prove that $C_s$ is continuous. Moreover, we choose  $\varphi_{\{s,t\}}:H_s\to H_t$
with $s<t$ defined as follows. Let $E_{\{s,t\}}:H^1_P([0,s];\R^n)\to H^1_P([0,t];\R^n)$ be the map that
carries an element of $H^1_P([0,s];\R^n)$ to its extension to zero in $[s,t]$; this gives
an element in $H^1_P([0,t];\R^n)$.
Then $\varphi_{\{s,t\}}(V)=\Phi_t^{-1}\cdot E_{\{s,t\}}\cdot\Phi_s(V)$. It is a straightforward computation  to verify that $\varphi_{\{s,t\}}$ satisfies the hypothesis in Proposition \ref{thm:MIT}. By Proposition~\ref{c1} the path $s\to\Hi_P(s)$ is smooth, and by \cite[Proposition 4.9]{BenPicpre} this implies the continuity as closed subspaces in the sense of Section~\ref{abstractMorse}; by Proposition~\ref{finiteindex} the symmetric bilinear forms $C_s$ are essentially positive and by Lemma~\ref{positivedefinite} the initial contribution is always zero, so that the index theorem follows for $\hat{I}_1$ and of course for $I=I_1$ restricted to ${\mathcal W}_P(1)\times {\mathcal W}_P(1)$ (we observe that ${\mathcal W}_P(1)=\Hi_P(1)$). By Proposition~\ref{thm:prop2.6} the dimensions of the kernel of $I_s$ restricted to ${\mathcal W}_P(s)\times {\mathcal W}_P(s)$ coincides with the number of $(P,Y)$-pseudo focal points counted with multiplicity.
\end{proof}
\subsection{The case of two variable endpoints.}\label{twoendopoints}
 We will use the idea of \cite[Theorem II.6]{PiTau99} to extend the Morse index theorem to the situation in that the two endpoints are variable. In the context of Morse-Sturm systems we have to add to the initial data $(g,R,Y,P,S)$ (after Eq. \eqref{indexform}) another subspace $Q$ of $\R^n$ and a $g$-symmetric linear map
$S^Q:Q\to Q$. Moreover, we rename $S$ as $S^P$. Thus, we have the initial data $(g,R,Y,P,Q,S^P,S^Q)$ and the index form
\begin{align}\label{indexform2ends}
I_{\{P,Q\}}(V,W)=\int_0^1& \big[g(V', W')+g\big(R(t)[V],W\big)\big]\,\de t\nonumber\\&
+g\big(S^Q[V(1)],W(1)\big)-g\big(S^P[V(0)],W(0)\big),
\end{align}
defined for $V$ and $W$ in $H^1_{\{P,Q\}}([0,1];\R^n)$, that is, the functions $V$ in $H^1([0,1];\R^n)$, such that $V(0)\in P$ and $V(1)\in Q$. Denote
\begin{equation*}
\Hi_{\{P,Q\}}=\big\{V\in H^1_{\{P,Q\}}([0,1];\R^n)\,:\, g(V'(t),Y(t))- g(V(t),Y'(t))=0\big\}
\end{equation*}
and
\begin{equation*}
\Hi^*_{\{P,Q\}}=\big\{V\in H^1_{\{P,Q\}}([0,1];\R^n)\,:\, g(V'(t),Y(t))- g(V(t),Y'(t))=C_V\in\R\big\}.
\end{equation*}
Furthermore, let ${\mathcal J}^*_Q$ be the subspace of $P$-Jacobi fields contained in $\Hi^*_{\{P,Q\}}$, ${\mathcal J}_Q={\mathcal J}^*_Q\cap\Hi_{\{P,Q\}}$ and $F$ the symmetric bilinear form obtained as the restriction of $I_{\{P,Q\}}$ to ${\mathcal J}^*_Q$. By applying integration by parts we obtain that
\begin{equation*}
F(J_1,J_2)=g\big(S^Q[J_1(1)],J_2(1)\big)+g\big(J'_1(1),J_2(1)\big),
\end{equation*}
where $J_1$ and $J_2$ are in ${\mathcal J}^*_Q$. Finally,
\begin{equation*}
{\mathcal J}[t]=\{J(t)\in \R^n\, :\, \text{$J$ is a $P$-Jacobi field and $C_J=0$} \}
\end{equation*}
and
\begin{equation*}
{\mathcal J}^*[t]=\{J(t)\in \R^n\, :\, \text{$J$ is a $P$-Jacobi field} \}.
\end{equation*}
Adapting \cite[Theorem II.6]{PiTau99} to this situation we obtain the following.
\begin{theorem}\label{twoendsindextheorem}
Assume, that ${\mathcal J}[1]\supseteq Q$. Then, the index of $I_{\{P,Q\}}$
restricted to $\Hi_{\{P,Q\}}\times \Hi_{\{P,Q\}}$ is equal to the sum
of the index of $I|_{\Hi_P(1)\times \Hi_P(1)}$ and the index of $F|_{{\mathcal J}_Q\times {\mathcal J}_Q}$.
\end{theorem}
\begin{proof}
Denote ${\mathcal J}_0=\{J\in {\mathcal J}_
Q\, :\, J[1]=0\}$ and choose a complementary subspace ${\mathcal J}_1$ in ${\mathcal J}_Q$ such that ${\mathcal J}_Q={\mathcal J}_0\oplus {\mathcal J}_1$. As ${\mathcal J}[1]\supseteq Q$, we have that $\Hi_{\{P,Q\}}=\Hi_P(1)\oplus{\mathcal J}_1$ and it is easy to see that this decomposition is $I_{\{P,Q\}}$-orthogonal. It follows that the index of $I_{\{P,Q\}}|_{\Hi_{\{P,Q\}}\times \Hi_{\{P,Q\}}}$ is equal to the index of $I|_{\Hi_P(1)\times\Hi_P(1)}$ plus the index of $F|_{{\mathcal J}_1\times {\mathcal J}_1}$. The theorem follows from the observation that ${\mathcal J}_0$ is contained in the kernel of $F$, so that the index of $F|_{{\mathcal J}_1\times {\mathcal J}_1}$ coincides with the one of $F|_{{\mathcal J}_Q\times {\mathcal J}_Q}$.
\end{proof}
\subsection{The Morse index theorem in the geometrical setup.}
\label{sub:geomsetup}
As it was observed in Section \ref{stationary}, the index form associated to a geodesic in a Lorentzian manifold $(M,g)$ can be reduced to
the index form of a Morse-Sturm system. When we consider the energy functional defined in the manifold $\Omega_{\{\mathcal P,\mathcal Q\}}M$ of $H^1$-curves joining two given submanifolds $\mathcal P$ and $\mathcal Q$ of $M$, the critical points are geodesics $\gamma:[0,1]\to M$ orthogonal to $\mathcal P$ and $\mathcal Q$ in the endpoints. Furthermore, the associated index form is defined for $V$ and $W$ in $T_\gamma\big(\Omega_{\{\mathcal P,\mathcal Q\}}M\big)$ and it is given by
\begin{align}\label{index2}
I_{\{\gamma,\mathcal P, \mathcal Q\}}(V,W)=\int_0^1 &\big[g(V', W')+g\big(R(\dot\gamma,V)\dot\gamma,W\big)\big]\,\de t\nonumber\\&-g\big(S_{\dot\gamma(0)}^{\mathcal P}[V(0)],W(0)\big)
+g\big(S_{\dot\gamma(1)}^{\mathcal Q}[V(1)],W(1)\big),
\end{align}
where $S_{\dot\gamma(0)}^{\mathcal P}$ and $S_{\dot\gamma(1)}^{\mathcal Q}$ are the second fundamental forms of $\mathcal P$ and $\mathcal Q$ in the directions of $\dot\gamma(0)$ and $\dot\gamma(1)$ respectively. We observe that the tangent space to $\Omega_{\{\mathcal P,\mathcal Q\}}M$ in $\gamma$ can be described as
\begin{multline*}
T_\gamma\big(\Omega_{\{\mathcal P,\mathcal Q\}}M\big)=\{\text{$V$\, :\, $V$ is a $H^1$-vector field along $\gamma$,}\\ \text{ $V(0)\in T_{\gamma(0)}\mathcal P$ and  $V(1)\in T_{\gamma(1)}\mathcal Q$}\}.
\end{multline*}

 Assume that there exists a timelike Jacobi field $\mathcal Y$ along $\gamma$. In order to establish the Morse index theorem we consider the subspaces
\[\mathcal H^{\{\gamma,\mathcal P,\mathcal Q\}}_*=
\big\{V\in T_\gamma\big(\Omega_{\{\mathcal P,\mathcal Q\}}M\big)\,
:\, g(V',\mathcal Y)-g(V,\mathcal Y')=\text{const.}\big\}\]
and
\[\mathcal H^{\{\gamma,\mathcal P,\mathcal Q\}}_0=\big\{V
\in T_\gamma\big(\Omega_{\{\mathcal P,\mathcal Q\}}M\big)\, :\, g(V',\mathcal Y)-g(V,\mathcal Y')=0\big\}.\]
We observe that we just suppress $\mathcal Q$ in all the notations when it is a point.
We know that the index of $I_{\{\gamma,\mathcal P, \mathcal Q\}}$ given in \eqref{index2} restricted to $\mathcal H^{\{\gamma,\mathcal P,\mathcal Q\}}_*\times \mathcal H^{\{\gamma,\mathcal P,\mathcal Q\}}_*$ and $\mathcal H^{\{\gamma,\mathcal P,\mathcal Q\}}_0\times \mathcal H^{\{\gamma,\mathcal P,\mathcal Q\}}_0$ is finite. Moreover the difference between the two restrictions is $1$ or $0$, because $\mathcal H^{\{\gamma,\mathcal P,\mathcal Q\}}_0$ is a codimensional-one subspace of $\mathcal H^{\{\gamma,\mathcal P,\mathcal Q\}}_*$, and we call this difference $\varepsilon_{\{\gamma,\mathcal P,\mathcal Q\}}$.
If we fix a parallel orthonormal frame along $\gamma$, we can get the initial data $(g,R,Y,P,Q,S^P,S^Q)$ as the corresponding coordinate version of $(g,R(\dot\gamma,\cdot)\dot\gamma,\mathcal Y, T_{\gamma(0)}\mathcal P,T_{\gamma(1)}\mathcal Q,S_{\dot\gamma(0)}^{\mathcal P},S_{\dot\gamma(1)}^{\mathcal Q})$ in such a way that the index \eqref{indexform2ends} is obtained from \eqref{index2} when considering the coordinates in the parallel orthonormal frame. Obviously, $\mathcal P$-focal points of $\gamma$ are in correspondence with $P$-focal points of the data $(g,R,Y,P,Q,S^P,S^Q)$, so that we can bring the Morse index theorem for Morse-Sturm systems to the geometrical setup. We extend all the definitions related to $P$-Jacobi fields and $(P,Y)$-pseudo Jacobi fields of a Morse-Sturm system (see Subsections \ref{parallelfocal} and \ref{admisingu} 
and Definition \ref{pseudo}) to $\mathcal P$-Jacobi fields and $(\mathcal P,\mathcal Y)$-pseudo Jacobi fields in the obvious way, so as the Definition \ref{defpseudo} of  $(\mathcal P,\mathcal Y)$-pseudo focal points. Furthermore, we will use ${\mathcal J}_{\mathcal Q}$, ${\mathcal J}^*_{\mathcal Q}$ and $\mathcal F$ to denote the geometrical objects correponding to ${\mathcal J}_{Q}$, ${\mathcal J}^*_{Q}$ and $F$ defined in Subsection \ref{twoendopoints}, and the same notation to the geometrical counterpart of ${\mathcal J}[t]$ and ${\mathcal J}^*[t]$.
In the next results we assume that $\mathcal Y(0)$ is not orthogonal to $\mathcal P$ or that $\mathcal Y$ is singular or admissible.
\begin{theorem}\label{morsegeodesic}
Assume that $\mathcal J[1]\supseteq T_{\gamma(1)}\mathcal Q$. Then the Morse index of ${I_{\{\gamma,{\mathcal P}, \mathcal Q\}}}$ given in \eqref{index2} restricted to $\mathcal H^{\{\gamma,\mathcal P,\mathcal Q\}}_0\times \mathcal H^{\{\gamma,\mathcal P,\mathcal Q\}}_0$
coincides with the sum of the number  of $(\mathcal P,\mathcal Y)$-pseudo focal points  counted with multiplicity of the geodesic $\gamma$ and the index of ${\mathcal F}|_{{\mathcal J}_{\mathcal Q}\times {\mathcal J}_{\mathcal Q}}$.
\end{theorem}
\begin{proof}
It follows from Theorems~\ref{Morseindex} and \ref{twoendsindextheorem}.
\end{proof}
\begin{corollary}\label{conjugatecase}
The number of $(\mathcal P,\mathcal Y)$-pseudo focal points counted with multiplicity is finite.\qed
\end{corollary}
\begin{corollary}
The Morse index of $I_{\{\gamma,{\mathcal P}\}}$ restricted to ${\Hi^{\{\gamma,\mathcal P\}}_*\times \Hi^{\{\gamma,\mathcal P\}}_*}$ coincides or it is one unit larger than the number of $(\mathcal P,\mathcal Y)$-pseudo focal points counted with multiplicity of the geodesic $\gamma$.
\end{corollary}
When there is a singular timelike Jacobi field along the geodesic we can obtain a stronger Riemannian Morse index theorem.
\begin{theorem}\label{morsesingulargeodesic}
Assume that the timelike Jacobi field $\mathcal Y$ is singular, $\mathcal P$ is orthogonal to $\mathcal Y(0)$ and ${\mathcal J}^*[1]\supseteq T_{\gamma(1)}{\mathcal Q}$. Then the Morse index of ${I_{\{\gamma,{\mathcal P}, \mathcal Q\}}}$  in \eqref{index2} restricted to $\mathcal H^{\{\gamma,\mathcal P,\mathcal Q\}}_*\times \mathcal H^{\{\gamma,\mathcal P,\mathcal Q\}}_*$ coincides with the sum  of the number of $\mathcal P$-focal points  counted with multiplicity of the geodesic $\gamma$ and the index of ${\mathcal F}|_{{\mathcal J}^*_{\mathcal Q}\times {\mathcal J}^*_{\mathcal Q}}$.
\end{theorem}
\begin{proof}
It is enough to observe that when $\mathcal Y$ is singular and $\mathcal Y(0)$ is orthogonal to $\mathcal P$, then $\mathcal H^{\{\gamma,\mathcal P,\mathcal Q\}}_*=\mathcal H^{\{\gamma,\mathcal P,\mathcal Q\}}_0$ (this can be shown as in the end of Proposition \ref{c1}) and also $\mathcal J_{\mathcal Q}=\mathcal J^*_{\mathcal Q}$. Moreover, as $m(\mathcal Y)=0$, $(\mathcal P,\mathcal Y)$-pseudo focal points coincide with $\mathcal P$-focal instans, therefore the thesis follows from Theorem~\ref{morsegeodesic}.
\end{proof}
The last theorem gives the Morse index theorem for timelike geodesics by taking $\mathcal Y=\dot\gamma$. Furthermore it can be used to compute the Morse index of a horizontal geodesic in a static spacetime as we will see later.
\begin{corollary}
 If the geodesic $\gamma$ admits a singular timelike Jacobi field $\mathcal Y$, then there only exists a finite number of conjugate instants along the geodesic. Moreover, if $\mathcal P$ is orthogonal  to $\mathcal Y(0)$ there only exists a finite number of $\mathcal P$-focal instants along $\gamma$.\qed
\end{corollary}
\subsection{Static manifolds}
A Lorentzian manifold is said to be standard static if it can be expressed as a product $M_0\times\R$ endowed with a metric given by
\begin{equation}\label{standardmetric}
g(x,t)[(\xi,\tau),(\xi,\tau)]=g_0(x)[\xi,\xi]-\beta(x)\tau^2
\end{equation}
where $(x,t)\in M_0\times\R$, $(\xi,\tau)\in T_xM_0\times \R$, $g_0$ is a Riemannian metric in $M_0$ and $\beta$ a $C^\infty$ positive function in $M_0$. Standard static spacetimes are always stationary with Killing field given by ${\mathcal Y}=(0,1)$ and we can prove a Morse index theorem  for every horizontal geodesic (in the following horizontal will mean orthogonal to the fibers of the natural projection $\pi:M_0\times\R\to M_0$).
\begin{proposition}\label{staticcase}
Let $(M_0\times\R,g)$ be a standard static spacetime with $g$ as in \eqref{standardmetric}, $\gamma$ a horizontal geodesic in $M_0\times \R$, $\mathcal P$ a horizontal submanifold through $\gamma(0)$ and orthogonal to $\dot\gamma(0)$  and $I_{\{\gamma,\mathcal P\}}$ its related Morse index form. Then the index of $I_{\{\gamma,\mathcal P\}}$ restricted to $\Hi^{\{\gamma,\mathcal P\}}_*\times\Hi^{\{\gamma,\mathcal P\}}_*$ coincides with the number of $\mathcal P$-focal points of $\gamma$ counted with multiplicity. Moreover, this index coincides with the index of $I_{\{\gamma,\mathcal{P}\}}$ in the Riemannian manifold $M_0$.
\end{proposition}
\begin{proof}
 A horizontal geodesic $[0,1]\ni s\to(x(s),t_0)\in M_0\times\R$ is always contained in the totally geodesic hypersurface $M_0\times \{t_0\}$, so that by Remark~\ref{tothyp} we conclude that  $\mathcal Y=(0,1)$ is singular. By applying Theorem \ref{morsesingulargeodesic}, the first part of the thesis follows. For the last part, it is enough to observe that $\mathcal P\subset M_0\times \{t_0\}$ and $\mathcal P$-Jacobi fields coincide with the $\mathcal P$-Jacobi fields of $\gamma$ in $M_0$. In order to see this, we observe that $g(J'(0),\mathcal Y(0))-g(J(0),\mathcal Y'(0))=0$ and as $J(0)$ is tangent to $\mathcal P$ and $\mathcal Y'(0)$ linearly dependent with $\mathcal Y(0)$ we deduce  that $J'(0)$ is tangent to $M_0\times \{t_0\}$. As $J(0)$ and $J'(0)$ are tangent to $M_0\times\{t_0\}$ and this submanifold is totally geodesic, it can be deduced that $J$ is also a $\mathcal P$-Jacobi field in $M_0\cong M_0\times \{t_0\}$.
\end{proof}
\end{section}

\begin{section}{Evolution of the index functions and the distribution of focal and pseudo focal points}
\label{sec:distribution}
Let us finally discuss the question of distribution of focal points along a
geodesic in a stationary Lorentzian manifold. Let us consider the geometrical
setup introduced in Subsection~\ref{sub:geomsetup}; we will consider the only interesting case of
a \emph{spacelike} geodesic $\gamma$.
Let us use the following notations: for all $t\in\left]0,1\right]$,
we will denote by $\mu(t)$ the index of the bilinear form:
\[I_t(V,W)=\int_0^tg(V',W')+g\big(R(\dot\gamma,V)\dot\gamma,W\big)\,\mathrm ds-g\big(S^{\mathcal P}_{\dot\gamma(0)}[V(0)],W(0)\big)\]
on the space $\mathcal H_t$ of $H^1$-vector fields $V$ along $\gamma\vert_{[0,t]}$ satisfying
$V(0)\in T_{\gamma(0)}\mathcal P$, $V(t)=0$ and $g(V',\mathcal Y)-g(V,\mathcal Y')=C_V$ a. e. on $[0,t]$. 
In the following we will asume that $\mathcal Y(0)$ is not orthogonal to $\mathcal P$ or that $\mathcal Y$ 
is singular or admissible.
By $\mu_0(t)$ we will denote the index of $I_t$ on the one-codimensional subspace $\mathcal H_t^0$
of $\mathcal H_t$ consisting of vector fields $V$ for which the constant $C_V$ vanishes.
The functions $\mu$ and $\mu_0$ give us information on the distribution of focal and
pseudo focal instants.
It is worth recalling that a $\mathcal P$-focal instant $t_0\in\left]0,1\right]$ along $\gamma$
is said to be \emph{nondegenerate} if the restriction of the metric $g$ to the space:
\[\mathbb J[t_0]=\big\{J'(t_0):\text{$J$ is a $\mathcal P$-Jacobi field and $J(t_0)=0$}\big\}\]
is nondegenerate. Nondegenerate $\mathcal P$-focal instants are isolated in the set of
all $\mathcal P$-focal instants.
Using the theory developed in this paper and some results in the recent literature, we can now summarize
a few facts about the distribution of focal and pseudo focal instants along a geodesic.
\begin{itemize}
\item[(a)] $0\le\mu(t)-\mu_0(t)\le 1$ for all $t\in\left]0,1\right]$. This follows from the fact
that the subspace $H_t^0\subset H_t$ has codimension $1$.
\smallskip

\item[(b)] The function $\mu_0$ is nondecreasing in $\left]0,1\right]$. This follows from the
fact that, for $t_0<t_1$ one has an injection $H_{t_0}^0\hookrightarrow H_{t_1}^0$ given
by extension to $0$ on $[t_0,t_1]$, and that an $I_{t_0}$-negative subspace of
$H_{t_0}^0$ is mapped by such injection onto an $I_{t_1}$-negative subspace of $H_{t_1}^0$.
\smallskip

\item[(c)] We cannot establish whether the function $\mu$ is nondecreasing. Note that there is no
natural injection $H_{t_0}\hookrightarrow H_{t_1}$ that preserves (or decreases the values of)
the index form, as in (b). Extension of vector fields to $0$ is
not allowed here, because of the constant $g(V',\mathcal Y)-g(V,\mathcal Y')$ is not zero.
\smallskip

\item[(d)] $\mu(t)$ is equal to the $\mathcal P$-Maslov index of $\gamma\vert_{[0,t]}$,
as proved in \cite{GiaMasPic01}.
\smallskip

\item[(e)] An instant $t_0\in\left]0,1\right[$ is a jump instant for the function $\mu_0$ if and only
if it is a $(\mathcal P,\mathcal Y)$-pseudo focal instant. This follows from the Morse index theorem (Theorem~\ref{morsegeodesic});
the value of the jump is precisely the multiplicity $\mathrm{mul}_0(t_0)$
of $t_0$ as a pseudo focal instant.
In particular, $\mu_0$ is constant on every interval that does not contain pseudo focal instants.
\smallskip

\item[(f)] If an instant $t_0\in\left]0,1\right[$ is a jump instant for $\mu$, then $t_0$ is a focal instant
along $\gamma$. This follows from the main result of \cite{GiaMasPic01}, since the $\mathcal P$-Maslov
of $\gamma$ has jumps only at the focal instants. Thus, $\mu$ is constant on every interval that does
not contain focal instants. The contribution to the index function $\mu$ given by a nondegenerate
$\mathcal P$-focal instant is given by the signature of the restriction of $g$ to the space
$\mathbb J[t_0]$. It is not known whether in the stationary case such contribution may be null or negative.
\smallskip

\item[(g)] The set of focal instants is closed, and contained in $\left]\varepsilon,1\right]$ for some
$\varepsilon>0$.
\smallskip

\item[(h)] We cannot establish whether all focal instants give a contribution to $\mu$.
Let us call \emph{effective} those focal instants that do determine a jump of the function $\mu$.
The jump of the function $\mu$ at an effective focal instant $t_0\in\left]0,1\right[$
is in absolute value less than or equal to the multiplicity $\mathrm{mul}(t_0)$ of $t_0$ as a focal instant.
Note that the first effective focal instant $t_0\in\left]0,1\right]$ must give a positive contribution
to $\mu$, because $\mu\ge0$.
\smallskip

\item[(i)] If $\mathrm{mul}(t_0)>1$, then $\mathrm{mul}_0(t_0)>0$, i.e., a focal point of multiplicity
larger than $1$ is pseudo focal. More precisely, $\mathrm{mul}_0(t_0)\ge\mathrm{mul}(t_0)-1$.
This follows from the fact that the space of $\mathcal P$-Jacobi fields
$J$ along $\gamma$ vanishing at $0$ and satisfying $C_J=g(J',\mathcal Y)-g(J,\mathcal Y')=0$ form a subspace of codimension $1$ of
the space of all $\mathcal P$-Jacobi fields  along $\gamma$ vanishing at $t_0$. This implies
that focal points with multiplicity larger than one do not accumulate.
\smallskip

\item[(j)] If $\mathrm{mul}_0(t_0)>1$, then $\mathrm{mul}(t_0)>0$, i.e., pseudo focal instants with
multiplicity larger than $1$ are focal. More precisely, $\mathrm{mul}(t_0)\ge\mathrm{mul}_0(t_0)-1$.
Namely, assume $\mathrm{mul}_0(t_0)=k>1$ and let $J_1,\ldots,J_k$ be a basis of $(\mathcal P,\mathcal Y)$-pseudo
Jacobi fields satisfying $J_s(t_0)=0$, with $J_s''-R[J_s]=\lambda_s m(Y)$
for all $s=1,\ldots,k$. Assume $\lambda_k\ne0$ (if all the $\lambda_s$ vanish, then all the $J_s$ are
$\mathcal P$-Jacobi fields vanishing at $t_0$, hence $\mathrm{mul}(t_0)\ge k$).
Then, the vector fields
$W_s=\lambda_kJ_s-\lambda_sJ_k$, $s=1,\ldots,k-1$, are linearly independent $\mathcal P$-Jacobi fields
vanishing at $t_0$, thus $\mathrm{mul}(t_0)\ge k-1$.
\smallskip

\item[(k)] If $t_1<t_2<1$ are pseudo focal instants, then there exists one (effective)
focal instant in the interval $[t_1,t_2]$. Namely, if there were no effective focal instant
in $[t_1,t_2]$, then $\lim\limits_{t\to t_2^+}\mu_0(t)\ge2+\lim\limits_{t\to t_1^-}\mu_0(t)\ge
1+\limsup\limits_{t\to t_1^-}\mu(t)=1+\liminf\limits_{t\to t_2^+}\mu(t)$, which gives a contradiction with
the inequality $\mu_0(t)\le\mu(t)$.
\smallskip

\item[(l)] It is not clear whether effective focal instants are isolated.
However, if $t_1<t_2\le1$ are \emph{consecutive} focal instants, i.e., there is no focal instant in
$\left]t_1,t_2\right[$, and if they give a positive contribution to $\mu$, then there exists one
pseudo focal instant in the interval $[t_1,t_2]$. Namely, if there were no pseudo focal instant
in $[t_1,t_2]$, then $\liminf\limits_{t\to t_2^+}\mu(t)\ge2+\limsup\limits_{t\to t_1^-}\mu(t)\ge
2+\lim\limits_{t\to t_1^-}\mu_0(t)=2+\lim\limits_{t\to t_2^+}\mu_0(t)$, which gives a contradiction with
the inequality $\mu(t)\le\mu_0(t)+1$. 
\smallskip

\item[(m)] If $t_0\in\left]0,1\right]$ is the first pseudo focal point, then there exists
an effective focal instant $t_1\in\left]0,t_0\right]$ that gives a positive contribution to $\mu$.
For, otherwise it would be $\lim\limits_{t\to t_0^+}\mu_0(t)=
\mathrm{mul}_0(t_0)>0=\liminf\limits_{t\to t_0^+}\mu(t)$, which contradicts the inequality $\mu_0(t)\le\mu(t)$.
\end{itemize}
\end{section}

\appendix

\begin{section}{Continuity and weak continuity of  families of subspaces}
\label{app:A}
We will discuss here a simple result showing that the abstract Morse index theorem discussed in
this paper and a similar result by Uhlenbeck (see \cite[Theorem 1.11]{Uhl73}) are in fact independent.
Recall that in \cite[Theorem 1.11]{Uhl73} it is considered an increasing family of
closed subspaces of a Hilbert space, satisfying the assumption
below.

Let $H$ be a Hilbert space and let $(H_s)_{s\in[a,b]}$ be a family of closed
subspaces of $H$; denote by $P_s:H\to H$ the orthogonal projection onto $H_s$.
We say that $(H_s)$ is \emph{continuous} if the map $s\mapsto P_s\in {\mathcal L}(H)$
is continuous in the operator norm topology.
A \emph{weaker} notion of continuity can be introduced by considering the \emph{strong operator
topology} (SOT) of ${\mathcal L}(H)$. Recall that this topology is the locally convex topology defined
by the family of semi-norms $T\mapsto\Vert T\xi\Vert$, where $\xi\in H$;
in other words, a net $T_\alpha$ converges to $T$ in the SOT if $T_\alpha\xi\to T\xi$ for
all $\xi\in H$. We say that $(H_s)$ is \emph{weakly continuous} if $s\mapsto P_s$ is
SOT-continuous.
\begin{lemma}\label{thm:lemcontinuity}
Assume that the family $(H_s)$ is non decreasing, i.e., $H_s\subset H_t$ when $s\le t$.
Then:
\begin{enumerate}
\item\label{itm:uno} $(H_s)$ is continuous if and only if it is constant.
\item\label{itm:due} $(H_s)$ is weakly continuous if and only if the following holds:
\begin{equation}
\label{eq:equivweakcont}
\overline{\bigcup_{s<t}H_s}=H_t\quad\forall\,t>a,\qquad\text{and}\qquad H_t=\bigcap_{s>t}H_s,\quad\forall\,t<b.
\end{equation}
\end{enumerate}
\end{lemma}
\begin{proof}
If $H_s\subset H_t$, then $P_t-P_s$ is the orthogonal projection onto the space $H_s^\perp\cap H_t$.
If $H_s\ne H_t$, then $H_s^\perp\cap H_t\ne\{0\}$, thus $\Vert P_t-P_s\Vert=1$, and $(H_s)$ can
only be continuous if $H_s=H_t$ for all $s,t$, which proves \eqref{itm:uno}.

In order to prove \eqref{itm:due}, assume that \eqref{eq:equivweakcont} holds, and fix $t>a$. Since the family $(H_s)$ is
nondecreasing, it is easy to see that, given $\xi\in H_t^\perp$, then $\lim\limits_{s\to t^-}P_s\xi=0=P_t\xi$.
If $\xi\in H_t$, then by the first equality in \eqref{eq:equivweakcont} for all $r<t$ there exists $\xi_r\in H_r$ such that $\lim\limits_{r\to t^-}\xi_r=\xi$.
Choose arbitrary $\varepsilon>0$ and let $r_0<t$ be such that $\Vert\xi_{r_0}-\xi\Vert<\varepsilon$;
then, for all $s\in\left[r_0,t\right[$ one has $P_s\xi_{r_0}=\xi_{r_0}$, and therefore:
\[\Vert P_s\xi-\xi\Vert\le\Vert P_s\xi-P_s\xi_{r_0}\Vert+\Vert P_s\xi_{r_0}-\xi\Vert\le\Vert P_s\Vert\Vert\xi_{r_0}-\xi\Vert+\Vert\xi_{r_0}-\xi\Vert<2\varepsilon.\]
This shows that $\lim\limits_{s\to t^-}P_s\xi=P_t\xi=\xi$,
and we have thus proven that for a nondecreasing family, the first equality in \eqref{eq:equivweakcont} implies the SOT left-continuity
of $P_s$.

Consider now the family $(K_s)$ of closed subspaces of $H$ given by $K_s=H_s^\perp$; this is a non increasing family of
subspaces, and the second equality in \eqref{eq:equivweakcont} is equivalent\footnote{The two equalities in \eqref{eq:equivweakcont}
are \emph{dual} to each other with respect to the operation of taking orthogonal complements.} to $\overline{\bigcup_{s>t}K_s}=K_t$.
By a totally analogous argument, the family of orthogonal projections $Q_s=1-P_s$ onto $K_s$ is SOT \emph{right-}continuous;
thus, $P_s$ is also right continuous.

Conversely, assume that $P_s$ is SOT continuous at $t$. Then, for all $\xi\in H_t$, $\lim\limits_{s\to t^-}P_s\xi=P_t\xi=\xi$.
Set $\xi_s=P_s\xi\in H_s$, so that $\lim\limits_{s\to t^-}\xi_s=\xi$, hence the first equality in \eqref{eq:equivweakcont} holds.
By duality, the SOT continuity of the projections $Q_s=1-P_s$ implies that also the second equality in \eqref{eq:equivweakcont} holds.
\end{proof}
\end{section}

\end{document}